\documentclass[a4paper]{amsart}
\usepackage[ps2pdf]{hyperref}
\usepackage{amsmath,amssymb}
\usepackage{bbold}
\usepackage{psfrag}
\usepackage{graphicx}
\usepackage{leftidx}
\usepackage[all,poly]{xy}
\usepackage[latin1]{inputenc}
\usepackage[dvipsnames]{xcolor}

\newtheorem{thm}{Theorem}[section]

\newtheorem{lem}[thm]{Lemma}
\newtheorem*{claim}{Claim}

\theoremstyle{definition}

\newcommand{\co}{\colon}
\newcommand{\id}{\mathrm{id}}

\newcommand{\opp}{\mathrm{op}}

\newcommand{\cc}{\mathcal{C}}

\newcommand{\zz}{\mathcal{Z}}

\newcommand{\RR}{\mathbb{R}}

\newcommand{\XX}{\mathbf{X}}
\newcommand{\YY}{\mathbf{Y}}

\newcommand{\Ker}{\mathrm{Ker}}

\newcommand{\kk}{\Bbbk}
\newcommand{\kt}{$\Bbbk$\nobreakdash-\hspace{0pt}}

\newcommand{\ti}{\mbox{-}\,}

\newcommand{\R}{\mathbb{R}}
\newcommand{\un}{\mathbb{1}}

\newcommand{\Int}{\mathrm{Int}}
\newcommand{\Reg}{\mathrm{Reg}}

\newcommand{\End}{\mathrm{End}}
\newcommand{\Hom}{\mathrm{Hom}}

\newcommand{\tr}{\mathrm{tr}}

\newcommand{\rank}{\mathrm{rank}}

\newcommand{\inv}{\mathbb{F}}

\newcommand{\lev}{\mathrm{ev}}

\newcommand{\rev}{\widetilde{\mathrm{ev}}}
\newcommand{\lcoev}{\mathrm{coev}}
\newcommand{\rcoev}{\widetilde{\mathrm{coev}}}
\newcommand{\ldual}[1]{#1^{*}}

\newcommand{\coul}[1]{\mathrm{Col}(#1)}
\newcommand{\scaledraw}[1]{A}
\newcommand{\scaleraisedraw}[2]{A}
\newcommand{\rsdraw}[3]{\raisebox{-#1\height}{\scalebox{#2}{\includegraphics{#3.eps}}}}
\newcommand{\labela}{\renewcommand{\labelenumi}{{\rm (\alph{enumi})}}}
\newcommand{\labeli}{\renewcommand{\labelenumi}{{\rm (\roman{enumi})}}}

\begin{document}

\title[On   3-dimensional Homotopy Quantum Field Theory, I]{On   3-dimensional Homotopy Quantum\\ Field Theory, I}
\author[V. Turaev]{Vladimir Turaev}
 \address{%
 Vladimir Turaev\newline
  \indent            Department of Mathematics, \newline
\indent  Indiana University \newline
                     \indent Bloomington IN47405 \newline
                     \indent USA \newline
\indent e-mail: vtouraev@indiana.edu}
\author[A. Virelizier]{Alexis Virelizier}
\address{%
 Alexis Virelizier\newline
     \indent         Department of Mathematics, \newline
\indent    University of Montpellier II\newline
                     \indent 34095 Montpellier Cedex 5 \newline
                     \indent France \newline
\indent e-mail:  virelizi@math.univ-montp2.fr} 
\date{\today}

\begin{abstract}
Given a discrete group $G$ and a spherical $G$-fusion category whose
neutral component has invertible dimension,   we use the state-sum
method to construct a 3-dimensional Homotopy Quantum Field Theory
with target the Eilenberg-MacLane space $K(G,1)$.
\end{abstract}
\maketitle

\setcounter{tocdepth}{1} \tableofcontents

\section{Introduction}\label{sec-Intro}

Homotopy Quantum Field Theory (HQFT)
is a branch of quantum topology concerned with  maps from manifolds
to a fixed target space. The aim   is to define and to study
homotopy invariants of such maps using  methods of quantum topology.
A formal notion of an HQFT was introduced in the monograph
\cite{Tu1} which treats in detail the 2-dimensional case. The
present paper focuses on 3-dimensional HQFTs with target the
Eilenberg-MacLane space $K(G,1)$ where $G$ is a discrete group.
These HQFTs  generalize more familiar 3-dimensional Topological
Quantum Field Theories (TQFTs) which correspond to the case $G=1$.

Two fundamental constructions of 3-dimensional  TQFTs  are due   to
Reshetikhin-Turaev   and
  Turaev-Viro.  The RT-construction may be viewed as a
mathematical realization of Witten's Chern-Simons TQFT. The
TV-construction is closely related to the Ponzano-Regge state-sum
model for   3-dimensional quantum gravity.   The Turaev-Viro
\cite{TV}
  state sum approach, as generalized by Barrett and Westbury
\cite{BW} (see also \cite{GK}), derives TQFTs from spherical fusion
categories.  In this paper, we extend the state sum approach to the
setting of HQFTs. Specifically, we show that any spherical
$G$-fusion category $\cc$ (of invertible dimension) gives rise to a
3-dimensional HQFT $\vert \cdot\vert_\cc$ with target $K(G,1)$.

As in the above mentioned papers, we represent 3-manifolds by their
skeletons. The maps to $K(G,1)$ are represented by certain labels on
the faces of the skeletons. The HQFTs are obtained by taking
appropriate state sums on the skeletons.  In distinction to the
earlier papers, we entirely avoid the use of $6j$-symbols and allow
non-generic skeletons, i.e., skeletons with edges incident to $\geq
4$ regions. In the case $G=1$ this approach was introduced in
\cite{TVi}.

This paper is the first in a series of papers in which we will
establish the following further results. We will show that the
Reshetikhin-Turaev surgery method also extends to HQFTs: any $G$-modular category
$\cc$ determines a 3-dimensional HQFT~$\tau_\cc$. We
will   generalize  the center construction for categories to
$G$-categories and show that the $G$-center $\zz_G(\cc)$ of a
spherical $G$-fusion category $\cc$ over an algebraically closed
field of characteristic zero is a $G$-modular category. Finally, we
will show that under these assumptions on $\cc$, the HQFTs
  $\vert
\cdot\vert_\cc$ and $\tau_{\zz_G(\cc)}$ are isomorphic. This theorem
is non-trivial already for $G=1$. In this case it was first
established by the present authors in \cite{TVi} and somewhat later
but independently by   Balsam and  Kirillov  \cite {KB}, \cite{Ba1},
\cite{Ba2}. The case of an arbitrary $G$ is considerably   more
difficult; it will be treated in the sequel.

The content of this paper is as follows. We   recall the notion of a
3-dimensional  HQFT in Section \ref{sec-HQFT}. Then we    discuss
various classes of monoidal categories and in particular $G$-fusion categories (Sections \ref{sec-prelimoncat} and
\ref{sec-Gfusioncat}). In Section \ref{sec-multimodulesandgraphs} we
consider symmetrized multiplicity modules in categories and
invariants of planar graphs needed for our state sums. In Section
\ref{sec-skeletonsstatesums} we discuss skeletons of 3-manifolds and
presentations of maps to $K(G,1)$ by labelings of skeletons. We use
these presentations in Section \ref{state-suminvariabnts ofclosed}
to derive from any   $G$-fusion category a numerical invariant  of
maps from closed 3-manifolds to $K(G,1)$. In the final Section
\ref{sec-TQFTfin} we extend these numerical invariance to an  HQFT
with target $K(G,1)$.  In the appendix we briefly discuss
push-forwards of categories and HQFTs.

Throughout the paper,  we fix  a (discrete) group $G$ and   an
Eilenberg-MacLane  space  $\XX$ of type $K(G,1)$ with base point
$x$.  Thus, $\XX$ is a   connected aspherical CW-space such that
$\pi_1(\XX,x)=G$. The symbol  $\kk$ will denote a commutative ring.

\section{Three-dimensional HQFTs}\label{sec-HQFT}

  We recall    following \cite{Tu1}
 the definition of a $3$-dimensional Homotopy
Quantum Field Theory (HQFT)  with target $\XX=K(G,1)$. Warning:  our
terminology here is adapted to the 3-dimensional case and differs
  from that  in \cite{Tu1}.

\subsection{Preliminaries on  $G$-surfaces and $G$-manifolds}\label{Preliminaries on  TQFTs}
 A topological space  $\Sigma$ is
  {\it pointed}   if   every  connected   component  of $\Sigma$    is provided with
a base point.  The set of base points of $\Sigma$ is denoted
$\Sigma_\bullet$. By a  {\it   $G$-surface}  we mean  a pair (a
pointed closed oriented smooth surface $\Sigma$, a homotopy class of
maps $f\co (\Sigma, \Sigma_\bullet)\to (\XX,x)$).     Reversing
orientation in a $G$-surface $(\Sigma,f)$, we obtain  a $G$-surface
$(-\Sigma, f)$. By convention, an empty set is   a $G$-surface with
unique orientation.

By a   {\it   $G$-manifold}  we    mean  a pair (a compact oriented
smooth $3$-dimensional manifold $M$ with pointed  boundary,
 a homotopy class of maps  $f\co  (M, (\partial
M)_\bullet) \to (\XX,x)$).  The manifold $M$ itself is not required
to be
  pointed. The boundary $(\partial M, f \vert_{\partial M} )$  of  a $G$-manifold $(M,f)$  is a $G$-surface.
We use  the \lq\lq outward vector first" convention  for the induced
orientation of the boundary: at any point of $ \partial M $ the
given orientation of $M$ is determined by the tuple (a tangent
vector directed outward, a basis in the tangent space of $\partial
M$ positive with respect to the induced orientation). A $G$-manifold
$M$ is {\it closed} if $\partial M=\emptyset$ (in this case
$(\partial M)_\bullet=\emptyset$).

Any $G$-surface $(\Sigma,f)$ determines the {\it cylinder
$G$-manifold}  $ (\Sigma\times [0,1], \overline f)$, where
$\overline f\colon \Sigma\times [0,1]\to \XX $ is the composition of
the projection to $  \Sigma$ with $f$. Here $\Sigma\times [0,1]$ has
the base points $\{(m,0), (m,1) \, \vert\, m\in \Sigma_\bullet\}$
and is oriented so that
 its oriented boundary is $  (-\Sigma\times \{0\}) \amalg  (\Sigma\times
 \{1\})$.

Disjoint unions of $G$-surfaces (respectively $G$-manifolds) are
$G$-surfaces (respectively $G$-manifolds) in the obvious way.  A
{\it  $G$-homeomorphism  of $G$-surfaces}  $ (\Sigma,f)\to
(\Sigma',f')$ is an orientation preserving diffeomorphism $g\colon
\Sigma\to \Sigma'$ such that $g(\Sigma_\bullet)=\Sigma'_\bullet$ and
$f =f'g  $. A {\it $G$-homeomorphism of $G$-manifolds} $ (M ,f)\to
(M' ,f') $ is an orientation preserving diffeomorphism $g\colon M
\to M' $ such that $g((\partial M)_\bullet)=(\partial M')_\bullet$
and $f =f'g $. In both cases, the equality $f =f'g$ is understood as
an equality of homotopy
  classes of maps.

  For brevity, we shall
usually omit the maps to $\XX$ from the notation for $G$-surfaces
and $G$-manifolds.

\subsection{The category of $G$-cobordisms}\label{category of cobs}  We
  define a  category of {\it 3-dimensional $G$-cobordisms}
$\mathrm{Cob}^G=\mathrm{Cob}_3^G$.
   Objects of $\mathrm{Cob}^G$ are $G$-surfaces.
   A morphism   $\Sigma_0 \to \Sigma_1$ in $\mathrm{Cob}^G$ is represented by  a pair (a $G$-manifold $M$,  a $G$-homeomorphism $ h\colon
    (-\Sigma_0) \sqcup\Sigma_1 \simeq  \partial M)$. We   call
    such pairs {\it $G$-cobordisms with bases $\Sigma_0 $ and $
    \Sigma_1$}.
 Two $G$-cobordisms $(M, {h}\co  (-\Sigma_0) \sqcup\Sigma_1
\to \partial M )$ and $(M', {h}' \co  (-\Sigma_0) \sqcup\Sigma_1 \to
\partial M')$ represent the same morphism   if
there is a $G$-homeomorphism  $g\co  M \to M'$ such that
${h}'=g{h}$. The identity morphism of a  $G$-surface $\Sigma$ is
represented by the cylinder $G$-manifold $ \Sigma \times [0,1]$ with
  tautological identification of the
boundary with $(-\Sigma) \sqcup\Sigma$. Composition of morphisms in
$\mathrm{Cob}^G$ is defined through gluing of $G$-cobordisms:    the
composition of morphisms $(M_0, {h}_0) \co \Sigma_0 \to \Sigma_1$
and $(M_1, {h}_1) \co \Sigma_1 \to \Sigma_2$  is    represented by
the $G$-cobordism  $(M,{h})$, where $M$ is the $G$-manifold obtained
by gluing
  $M_0$ and $M_1$ along ${h}_1 {h}_0^{-1} \co
{h}_0(\Sigma_1)   \to {h}_1({\Sigma_1})$ and $${h}={h}_0
\vert_{\Sigma_0} \sqcup {h}_1 \vert_{\Sigma_2} \co (-\Sigma_0)
\sqcup\Sigma_2 \simeq  \partial M.$$   The given maps $(M_i,
(\partial M_i)_\bullet) \to (\XX, x)$, $i=0,1$ may be chosen in
their homotopy classes to agree on ${h}_0(\Sigma_1)   \approx
{h}_1({\Sigma_1})$ and define thus a map   $(M, (\partial
M)_\bullet)\to (\XX,x)$. The asphericity of $\XX$ ensures that the
homotopy class of this map is well defined.

  The category $\mathrm{Cob}^G$ is   a symmetric monoidal category
with tensor product   given by   disjoint union of $G$-surfaces and
$G$-manifolds.
The unit object of $\mathrm{Cob}^G$ is the empty $G$-surface
$\emptyset$.

\subsection{HQFTs}\label{HQFTs} Let $\mathrm{vect}_\kk$ be the category
of finitely generated projective $\kk$-modules and
$\kk$-homomorphisms.  It is a symmetric monoidal category with
standard tensor
  product and unit object $\kk$. A  (3-dimensional)   {\it Homotopy Quantum Field Theory (HQFT)   with target $\XX$}   is a symmetric monoidal functor
  $Z\co \mathrm{Cob}^G \to \mathrm{vect}_\kk$. In particular,   $Z( \Sigma  \sqcup \Sigma') \simeq Z(
\Sigma ) \otimes Z( \Sigma') $ for any  $G$-surfaces $\Sigma,
\Sigma'$, and similarly    for  morphisms. Also, $Z(\emptyset)\simeq\kk$. We refer to \cite{ML1} for a detailed definition of a strong monoidal functor.


Every $G$-manifold $M$ determines two  morphisms $\emptyset \to
\partial M$ and   $  -\partial M \to  \emptyset$ in
$\mathrm{Cob}^G$. The associated   homomorphisms $\kk\simeq Z(\emptyset)
\to Z (\partial M)$ and $ Z (-\partial M)\to  Z(\emptyset)\simeq\kk$ are
denoted $Z(M, \emptyset,  \partial M)$ and $Z(M, -\partial M,
\emptyset)$, respectively. If $\partial M=\emptyset$, then $Z(M, \emptyset,  \partial M) \co \kk \to Z(\emptyset)\simeq\kk$ and $Z(M, -\partial M, \emptyset)\co  \kk\simeq Z(\emptyset) \to \kk$ are multiplication by the same element of $\kk$ denoted $Z(M)$.

The category  of $G$-cobordisms   $\mathrm{Cob}^G$ includes as
a subcategory  the category $\mathrm{Homeo}^G$ of $G$-surfaces and
their $G$-homeomorphisms considered up to isotopy (in the class of
$G$-homeomorphisms).
 Indeed,
 a  $G$-homeomorphism of $G$-surfaces  $g \co \Sigma \to \Sigma'$
determines  a morphism  $  \Sigma \to \Sigma' $ in $\mathrm{Cob}^G$
represented by the pair $(C =\Sigma' \times [0,1], h\co    (-\Sigma)
\sqcup\Sigma'  \simeq\partial C)$, where
  $h  (x ) =(g(x), 0) $  for $x\in \Sigma$ and    $h(x') =(x', 1)$ for $x'\in
  \Sigma'$. Isotopic $G$-homeomorphisms   give rise   to the same morphism in $\mathrm{Cob}^G$.
  The category $\mathrm{Homeo}^G$ inherits
  a structure
  of a symmetric monoidal category from that of $\mathrm{Cob}^G$.
    Restricting an HQFT $ Z\co    \mathrm{Cob}^G\to \mathrm{vect}_\kk$ to $\mathrm{Homeo}^G$,
     we obtain a symmetric monoidal functor  $ \mathrm{Homeo}^G\to \mathrm{vect}_\kk$.
     In particular,  $Z$ induces a $\kk$-linear
     representation of the \emph{mapping class group of a $G$-surface $\Sigma$}
      defined as   the group of isotopy classes of $G$-homeomorphisms   $\Sigma\to \Sigma$.

 For $G=\{1\}$, the space $\XX$ is just a point and without any
loss of information we may forget
    the   maps of surfaces and manifolds to $\XX$. We recover thus the familiar notion of a 3-dimensional
      TQFT.

\section{Preliminaries on monoidal   categories}\label{sec-prelimoncat}

In this section we recall several  basic definitions of the
theory of monoidal categories needed for the sequel.

\subsection{Conventions}\label{Conventions} The symbol $\cc$ will denote  a monoidal category with unit object~$\un$.
Notation  $X\in \cc$ will   mean    that $X$ is an object of $\cc$.
To simplify the formulas, we  will always pretend that   $\cc$   is
strict. Consequently, we  omit brackets in the tenor products and
suppress the associativity constraints $(X\otimes Y)\otimes Z\cong
X\otimes (Y\otimes Z)$ and the unitality constraints $X\otimes \un
\cong X\cong \un \otimes X$.    By the tensor product $X_1 \otimes
X_2 \otimes \cdots \otimes X_n$ of $n\geq 2$ objects $X_1,...,
X_n\in \cc$ we mean $(... ((X_1\otimes X_2) \otimes X_3) \otimes
\cdots \otimes X_{n-1}) \otimes X_n$.

\subsection{Pivotal  and spherical  categories}\label{pivotall}
Following \cite{Malt}, by a   \emph{pivotal}   category we mean
a monoidal category $\cc$ endowed with a rule which assigns to each
object $X\in \cc$  a \emph{dual object}~$X^*\in \cc$ and four
morphisms
\begin{align*}
& \lev_X \co X^*\otimes X \to\un,  \qquad \lcoev_X\co \un  \to X \otimes X^*,\\
&   \rev_X \co X\otimes X^* \to\un, \qquad   \rcoev_X\co \un  \to X^* \otimes X,
\end{align*}
satisfying the following   conditions:
\begin{enumerate}
  \renewcommand{\labelenumi}{{\rm (\alph{enumi})}}
  \item For every $X\in\cc$, the triple $(X^*, \lev_X,\lcoev_X)$ is a left
dual of $X$, i.e.,
$$
(\id_X \otimes \lev_X)(\lcoev_X \otimes \id_X)=\id_X \quad \text{and} \quad (\lev_X \otimes \id_{X^*})(\id_{X^*} \otimes \lcoev_X)=\id_{X^*};
$$
\item For every $X\in\cc$, the triple $(X^*, \rev_X,\rcoev_X)$ is a right
dual of~$X$, i.e.,
$$
(\rev_X \otimes \id_X)(\id_X \otimes \rcoev_X)=\id_X \quad \text{and} \quad (\id_{X^*} \otimes \rev_X)(\rcoev_X \otimes \id_{X^*})=\id_{X^*};
$$
\item For every morphism $f\co X \to Y$ in $\cc$,
the \emph{left dual}
$$
f^*= (\lev_Y \otimes  \id_{X^*})(\id_{Y^*}  \otimes f \otimes \id_{X^*})(\id_{Y^*}\otimes \lcoev_X) \colon Y^*\to X^*$$ is equal to the
 \emph{right  dual}
$$
f^*= (\id_{X^*} \otimes \rev_Y)(\id_{X^*} \otimes f \otimes \id_{Y^*})(\rcoev_X \otimes \id_{Y^*}) \colon Y^*\to X^*;$$
\item For all   $X,Y\in
\cc$, the  \emph{left  monoidal constraint}
$$(\lev_X  \otimes \id_{(Y \otimes X)^*})(\id_{X^*}  \otimes \lev_Y \otimes  \id_{(Y \otimes X)^*})(\id_{X^* \otimes Y^*}\otimes \lcoev_{Y \otimes
X})\colon \ldual{X} \otimes \ldual{Y} \to (Y \otimes X)^*$$ is
equal to the  \emph{right  monoidal constraint} $$(\id_{(Y
\otimes X)^*} \otimes \rev_Y)(\id_{(Y \otimes X)^*} \otimes
\rev_X \otimes \id_{X^*})(\rcoev_{Y \otimes X}\otimes \id_{X^*
\otimes Y^*}) \colon \ldual{X} \otimes \ldual{Y} \to (Y \otimes
X)^*;$$
\item  $\lev_\un=\rev_\un \colon \un^* \to \un$ (or, equivalently, $\lcoev_\un=\rcoev_\un\colon \un  \to \un^*$).
\end{enumerate}

 If $\cc$ is pivotal, then for any endomorphism $f$ of an object
$X\in \cc$, one   defines the {\it left} and  {\it right} traces
$$\tr_l(f)=\lev_X(\id_{\ldual{X}} \otimes f) \rcoev_X  \quad {\text {and}}\quad \tr_r(f)=  \rev_X( f \otimes
\id_{\ldual{X}}) \lcoev_X .$$ Both traces take values in $\End_\cc(\un)$ and are symmetric: $\tr_l
(gh)=\tr_l(hg)$ for any morphisms $g\co X\to Y$, $h\co Y\to X$ in $\cc$
and similarly
 for $\tr_r$. Also $\tr_l(f)=\tr_r( {f}^*)=\tr_l(f^{**})$ for any
 endomorphism $f$ of an object. The  {\it left} and  {\it right}
 dimensions of an object $X\in \cc$ are defined by
 $\dim_l(X)=\tr_l(\id_X)$ and $\dim_r(X)=\tr_r(\id_X)$. Clearly,
 $\dim_l(X)=\dim_r(X^*)=\dim_l(X^{**})$ for all $X$.

When $\cc$ is pivotal, we will suppress the duality constraints
$\un^* \cong \un$ and $X^* \otimes Y^*\cong (Y\otimes X)^* $. For
example, we will write $(f \otimes g)^*=g^* \otimes f^*$ for
morphisms $f,g$ in~$\cc$.

A pivotal category $\cc$ is \emph{spherical} if the left and right
traces of endomorphisms in $\cc$ coincide. Set then $\tr
(f)=\tr_l(f)=\tr_r(f)$ for any endomorphism  $f$ of an object of
$\cc$,   and $\dim (X)=\dim_l(X)=\dim_r(X)=\tr(\id_X)$ for any $X\in
\cc$.

 \subsection{Additive categories}
A category ${\mathcal C}$ is  \emph{$\kk$-additive} if the
$\Hom$-sets in ${\mathcal C}$ are modules over the
 ring $\kk$,  the  composition of morphisms of $\cc$ are \kt bilinear, and any finite family of objects of $\cc$ has a   direct sum
in~$\cc$.  Note that the direct sum of an empty family of objects is
a {\it null object},   that is, an object $\mathbf 0\in \cc$ such
that $\End_\cc (\mathbf 0)=0$.

A monoidal category is  \emph{$\kk$-additive} if it is
$\kk$-additive as a category and the monoidal product of morphisms
is \kt bilinear.

\subsection{Semisimple categories}\label{sect-semisimple-cat}
We call an object $U$ of a $\kk$-additive category $\cc$
\emph{simple} if $\End_\cc(U)$ is a free \kt module of rank 1 (and
so has the basis $\{\id_U\}$).     It is clear that   an object
isomorphic to a
 simple object is itself  simple. If $\cc$ is pivotal, then the dual of a   simple object of $\cc$ is  simple.

A \emph{split semisimple category (over $\kk$)} is a  $\kk$-additive
category $\cc$ such that
\begin{enumerate}
  \labela
\item each object of $\cc$ is a finite direct sum of  simple objects;
\item for any non-isomorphic  simple objects $i,j  $ of $\cc$, we have $\Hom_\cc(i,j)=0$.
\end{enumerate}
 Clearly, the Hom spaces in
such a $\cc$ are   free $\kk$-modules of finite rank. For  $X\in\cc$
and a  simple object $i \in \cc$, the  modules $H_X^i=\Hom_\cc(X,i)$
and $H_i^X=\Hom_\cc(i,X)$ have same   rank denoted $N^i_X $ and
called the \emph{multiplicity number}. The bilinear form $H_X^i
\times H_i^X\to \kk$ carrying   $(p\in H_X^i, q\in H_i^X)$ to $pq\in
\End_\cc (i)=\kk$ is non-degenerate. Note that if $i$ admits a
left or right dual $\hat{\imath}$, then $N_{i \otimes
X}^\un=N_{X}^{\hat{\imath}}=N_{X \otimes i}^\un$.

\subsection{Pre-fusion   and fusion   categories}\label{sect-pre-fusion-cat}
  A \emph{pre-fusion category (over $\kk)$} is a split semisimple
$\kk$-additive pivotal category $\cc$ such that the unit object
$\un$ is  simple. In such a   category, the map $\kk \to
\End_\cc(\un), k \mapsto k \, \id_\un$  is a \kt algebra isomorphism
which we use  to identify $\End_\cc(\un)=\kk$. The left and right
dimensions of any simple object of a pre-fusion category are
invertible (see, for example, Lemma 4.1 of \cite{TVi}).

If $I$ is a set of simple objects of pre-fusion category $\cc$ such that every
simple object of $\cc$ is isomorphic to a unique element of~$I$, then for any object $X$ of $\cc$, $N_X^i=0$ for all but a finite number of $i\in I$, and
\begin{equation}\label{eq-dim-mult-numbers}
\dim_l(X)=\sum_{i\in I} \dim_l(i) N_X^i, \qquad \dim_r(X)=\sum_{i\in I} \dim_r(i) N_X^i.
\end{equation}

A \emph{fusion category} is a pre-fusion category such that the
set of isomorphism classes of simple objects is finite. The
\emph{dimension} $\dim(\cc)$ of a   fusion   category $\cc$ is
defined by
$$\dim(\cc)=\sum_{i\in I} \dim_l(i)\dim_r(i) \in \kk,$$
where $I$ is a (finite)  set   of   simple objects of   $\cc$ such
that every simple object of $\cc$ is isomorphic to a unique element
of~$I$. The sum on the right-hand side   does not depend on the
choice of~$I$. Note that
  if   $\kk $ is an algebraically closed field of characteristic zero, then $\dim(\cc)\neq 0$, see \cite{ENO}.

\section{$G$-fusion categories}\label{sec-Gfusioncat}


\subsection{$G$-categories}\label{G-cat-deb} A \emph{$G$-graded category} or shorter a  \emph {$G$-category} is a
$\kk$-additive  monoidal category $\cc$ endowed with a system of
pairwise disjoint full $\kk$-additive subcategories $\{{\mathcal
C}_{g}\}_{{g}\in G}$ such that
\begin{enumerate}
  \labela
  \item
  each object $U \in {\mathcal C}$ splits as a  direct sum
$\oplus_{{g} } \, U_{g}$ where $U_{g}\in {\mathcal C}_{g}$ and
$g$ runs over a finite subset of $G$;

 \item if $U\in {\mathcal C}_{g}$ and  $V\in {\mathcal C}_{h}$, then
$U\otimes V\in {\mathcal C}_{{g}{h}}$;

\item if $U\in {\mathcal C}_{g}$ and $V\in
{\mathcal C}_{h}$ with ${g}\neq {h}$, then $\Hom_{\mathcal C}
(U,V)=0$;

\item the unit object $\un$ of $\cc$ belongs to $ {\mathcal C}_1$.

\end{enumerate}

Under these assumptions, we   write   ${\cc}=\oplus_{g} \,
{\cc}_{g}$. The category ${\mathcal C}_1$ corresponding to the
neutral element $1\in G$ is called the {\it neutral component}
of~${\mathcal C}$. Clearly, ${\mathcal C}_1$ is a $\kk$-additive
monoidal category.

 An object $X$ of a $G$-category ${\cc}=\oplus_{g} \,
{\cc}_{g}$ is   {\it homogeneous} if   $X\in \cc_g$ for some $g\in
G$. Such a $g$
 is then uniquely determined by $X$ and  denoted $|X|$. If
 two homogeneous objects $X, Y\in \cc $ are
 isomorphic, then either they are null objects or $|X|=|Y|$.


 A $G$-category  $\cc$ is \emph{pivotal} (resp.\ {\it spherical}) if it is pivotal (resp.\ spherical) as a
monoidal category. For such $\cc$ and   all $X \in \cc_g$ with $g\in
G$, we can and always do choose $X^*$ to be in  $ \cc_{g^{-1}}$.
 Note that if $\cc$ is   pivotal (resp.\ spherical), then so is $\cc_1$.

 A $G$-category  is \emph{pre-fusion} if it is pre-fusion as a
monoidal category. In particular, a  pre-fusion $G$-category is
supposed to be  pivotal.  In a pre-fusion $G$-category
$\cc=\oplus_{g \in G} \, \cc_g$, every simple object is isomorphic
to a simple object of $\cc_g$ for a unique $g\in G$. Moreover, for
all $g\in G$, each object of $\cc_{g}$ is a finite direct sum of
simple objects of $\cc_{g}$.

A  set $I$  of simple  objects of a pre-fusion $G$-category $\cc$ is
\emph{representative} if $\un \in I$, all elements of $I$ are
homogeneous, and every simple object of $\cc$ is isomorphic to
a unique element of~$I$. Any   such   set $I$ splits  as a disjoint
union $I=\amalg_{g\in G}\, I_g
 $  where $I_{g}$ is the set   of all elements of $I$
 belonging to $\cc_{g}$.

\subsection{$G$-fusion categories}\label{sect-fusion+} A \emph{$G$-fusion category} is a
pre-fusion $G$-category~$\cc$ such that   the set of isomorphism
classes of simple objects of~$\cc_g$   is finite and non-empty for
every $g\in G$. For $G=1$, we obtain the    notion of a fusion
category (see Section~\ref{sect-pre-fusion-cat}).   The neutral
component $\cc_1$ of a $G$-fusion category $\cc$ is a fusion
category. A $G$-fusion category is a fusion category if and only if
$G$ is finite.

%
%
%
%
%
%

 In the next statement we use the multiplicity numbers defined in Section
 \ref{sect-semisimple-cat}.

\begin{lem}\label{bubbleidentity} Let  $I=\amalg_{g\in G} I_g$ be  a  representative set of   simple objects of a $G$-fusion category~$\cc$.
Then
\begin{enumerate}
\labela
\item  For all $g\in G$,
$$\sum_{i \in I_g} \dim_l( i)\dim_r( i) =\dim(\cc_1).$$
\item For all $a,b,g \in G$, $U \in \cc_a$, and $ V \in \cc_b$,
\begin{equation*}
\sum_{k\in I_g,\, \ell\in I_{(agb)^{-1}}} \!\!\!\!\!\!\!\!\!\dim_l(k)\dim_l(\ell)
N_{U
\otimes k \otimes V \otimes \ell}^\un  =   \dim_r(U) \dim_r(V)   \dim(\cc_1) .
\end{equation*}
\end{enumerate}
\end{lem}
\begin{proof}
Let us prove (a). Pick   $k\in I_g$ and fix it till the end of
the argument. For every $i \in I_g$, the object $i \otimes k^* \in
\cc_1$ splits as a (finite) direct sum of simple objects
 necessarily belonging to  $\cc_1$. Every $j \in I_1$ occurs $N_{i \otimes
k^*}^j$ times in this sum. Then there is a family   of morphisms
$(p_i^{j,\alpha} \co i \otimes k^* \to j, q_i^{j,\alpha} \co j \to i
\otimes k^*)_{ \alpha \in A_{i,j}}$ such that $A_{i,j}$ has $N_{i
\otimes k^*}^j$ elements and $p_i^{j,\alpha}
q_i^{j,\alpha}=\delta_{\alpha,\beta} \,\id_j$ for all $\alpha,\beta
\in A_{i,j}$. This implies
\begin{equation}\label{eq-loc-dim1}
\id_{i \otimes k^*}=\sum_{j \in I_1,\, \alpha \in A_{i,j}} q_i^{j,\alpha}p_i^{j,\alpha}.
\end{equation}
For $i \in I_g$, $j \in I_1$, and $\alpha \in A_{i,j}$, set
\begin{gather*}
\psfrag{p}[Bc][Bc]{\scalebox{1}{$p_i^{j,\alpha}$}}
\psfrag{q}[Bc][Bc]{\scalebox{1}{$q_i^{j,\alpha}$}}
\psfrag{X}[Bl][Bl]{\scalebox{.9}{$k$}}
\psfrag{i}[Br][Br]{\scalebox{.9}{$j$}}
\psfrag{j}[Bl][Bl]{\scalebox{.9}{$i$}}
P_j^{i,\alpha}= \frac{\dim_r(i)}{\dim_r(j)}\, (\id_i \otimes \lev_k)(q_i^{j,\alpha} \otimes \id_k) =\frac{\dim_r(i)}{\dim_r(j)}\;\, \rsdraw{.45}{.9}{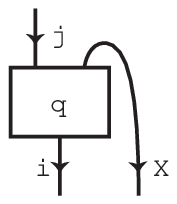} \co j \otimes k \to i, \\
\psfrag{p}[Bc][Bc]{\scalebox{1}{$p_i^{j,\alpha}$}}
\psfrag{q}[Bc][Bc]{\scalebox{1}{$q_i^{j,\alpha}$}}
\psfrag{X}[Bl][Bl]{\scalebox{.9}{$k$}}
\psfrag{i}[Br][Br]{\scalebox{.9}{$j$}}
\psfrag{j}[Bl][Bl]{\scalebox{.9}{$i$}}
Q_j^{i,\alpha}=(p_i^{j,\alpha} \otimes \id_k)(\id_i \otimes
\lcoev_k) = \, \rsdraw{.45}{.9}{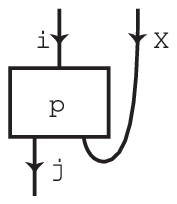} \co i \to j \otimes k.
\end{gather*}
For any $\alpha,\beta \in A_{i,j}$,
$$
P_j^{i,\alpha}Q_j^{i,\beta}=\frac{\tr_r(P_j^{i,\alpha}Q_j^{i,\beta})}{\dim_r(i)} \, \id_i =\frac{\tr_r(p_i^{j,\alpha}q_i^{j,\beta})}{\dim_r(j)}  \, \id_i
=\delta_{\alpha,\beta} \, \id_i.
$$
Note that  the set $A_{i,j}$ has $N_{i \otimes k^*}^j=N_{j \otimes
k}^i$ elements.  Thus, for every $j\in I_1$, the family
$(P_j^{i,\alpha},Q_j^{i,\alpha})_{i \in I_g, \alpha \in A_{i,j}}$
encodes a  splitting of $j \otimes k$ as a direct sum of simple
objects  of $\cc_g$. Hence
\begin{equation}\label{eq-loc-dim2}
\id_{j \otimes k}=\sum_{i \in I_g,\, \alpha \in A_{i,j}} Q_j^{i,\alpha}P_j^{i,\alpha}.
\end{equation}
Using \eqref{eq-loc-dim1}, the definition of  $P_j^{i,\alpha}$,
$Q_j^{i,\alpha}$, and   \eqref{eq-loc-dim2}, we obtain
\begin{gather*}
\sum_{i \in I_g} \dim_r( i) \dim_l( i) \dim_l(k) = \!\! \sum_{\substack{i
\in I_g , \, j \in I_1\\ \alpha \in A_{i,j}}} \!\! \dim_r( i)  \;
\psfrag{p}[Bc][Bc]{\scalebox{1}{$p_i^{j,\alpha}$}}
\psfrag{q}[Bc][Bc]{\scalebox{1}{$q_i^{j,\alpha}$}}
\psfrag{X}[Bl][Bl]{\scalebox{.9}{$k$}}
\psfrag{i}[Bl][Bl]{\scalebox{.9}{$j$}}
\psfrag{j}[Bl][Bl]{\scalebox{.9}{$i$}}
\rsdraw{.45}{.9}{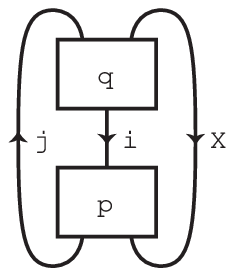}\\
= \!\!\!\!\!\!\!\!\!\!\!\sum_{i \in I_g, \,  j \in I_1, \, \alpha
\in A_{i,j}} \!\!\!\!\!\!\!\!\! \!\!\dim_r( j) \,
\tr_l(P_j^{i,\alpha}Q_j^{i,\alpha}) = \sum_{j \in I_1} \dim_r( j)
\!\!\!\!\!\! \sum_{i \in I_g, \,  \alpha \in A_{i,j}}
\!\!\!\!\!\!\tr_l(Q_j^{i,\alpha}P_j^{i,\alpha})\\ =\sum_{j \in I_1}
\dim_r( j) \, \tr_l(\id_{j \otimes k})=\sum_{j \in I_1} \dim_r( j)
\dim_l( j) \dim_l(k)=\dim(\cc_1)\dim_l(k).
\end{gather*}
We conclude using that $\dim_l(k)\in \kk$ is invertible.

Let us prove (b). Using \eqref{eq-dim-mult-numbers} and Claim (a) of
the lemma, we obtain
\begin{gather*}
\sum_{k\in I_g,\, \ell\in I_{(agb)^{-1}}}
\!\!\!\!\!\!\!\!\!\dim_l(k)\dim_l(\ell) N_{U \otimes k\otimes V
\otimes  \ell}^\un =\!\!\!\!\!\!\!\!\sum_{k\in I_g,\, \ell\in
I_{(agb)^{-1}}}
\!\!\!\!\!\!\!\!\! \dim_l(k) \dim_r(\ell^*) N_{U \otimes k\otimes V}^{\ell^*}\\
=\sum_{k\in I_g} \dim_l(k) \sum_{m\in I_{agb}}\!\dim_r(m) N_{U
\otimes k\otimes V}^{m}
=\sum_{k\in I_g} \dim_l(k) \dim_r(U \otimes k\otimes V)\\
= \dim_r(U)  \dim_r(V) \sum_{k\in I_g} \dim_l(k) \dim_r(k)=
\dim_r (U) \dim_r(V) \dim(\cc_1). \qedhere
\end{gather*}
 \end{proof}

\subsection{Example}\label{exaexa} Let   $\kk^*$ be  the group
of invertible elements of $\kk$ and let $\mathrm{vect}_G$ be the
category of  $G$-graded free $\kk$-modules of finite rank. It is
well known that every $\theta \in H^3(G;\kk^*)$ determines
associativity constraints on $\mathrm{vect}_G$  extending the usual
tensor product to a monoidal structure. Denote  the resulting
monoidal category by $\mathrm{vect}_G^\theta$. This category is
$G$-graded: $\mathrm{vect}_G^\theta=\oplus_{g \in G}\,
\mathrm{vect}_g^\theta$, where $\mathrm{vect}_g^\theta$ is the full
subcategory of modules of degree $g$. The module  $\kk $  viewed as
an object of $\mathrm{vect}_g^\theta$ is the unique simple object of
$\mathrm{vect}_g^\theta$, at least up to isomorphism. The usual
(co)evaluation morphisms define a pivotal structure on
$\mathrm{vect}_G^\theta$.     It is easy to check that
$\mathrm{vect}_G^\theta$ is s spherical $G$-fusion category.
Clearly,  $\dim(\mathrm{vect}_1^\theta)=1\in \kk$.

\subsection{Remark}\label{remarem}
Each  pre-fusion category $\cc$ has a {\it universal grading}
 defined as follows. Let ${\mathcal I}$ be the set of
isomorphism classes  of simple objects of~$\cc$. Note that gradings
of~$\cc$ by a group   (in the sense of Section~\ref{G-cat-deb})
bijectively correspond to maps $\varrho$ from ${\mathcal I}$ to this
group verifying $\varrho(X)\, \varrho(Y)=\varrho(Z)$ for all
$X,Y,Z\in {\mathcal I}$ such that $Z$ is a direct summand of $X
\otimes Y$. A grading $\partial\co {\mathcal I} \to \Gamma$ of $\cc$
by a group $\Gamma$ is {\it universal} if any grading      of~$\cc$
by any group can be uniquely expressed as the composition of
$\partial$ with a homomorphism from $ \Gamma $ to that group. To
construct a universal grading of $\cc$ set  $\Gamma={\mathcal
I}/\!\!\sim$, where $\sim$ is the weakest equivalence relation on
${\mathcal I}$ such that any two simple objects appearing as direct
summands of the same (finite) monoidal product of simple objects are
equivalent. The product, the unit, and the inverses in $\Gamma$ are
induced by the pivotal structure of $\cc$. The projection $
{\mathcal I} \to \Gamma={\mathcal I}/\!\!\sim$ is a universal
grading of $\cc$. The   group $\Gamma$ is called
  the {\it graduator} of $\cc$.
  For instance, the grading  of
$\mathrm{vect}_G^\theta$ in Example \ref{exaexa} is universal, and
the graduator is $G$.

\section{Multiplicity modules and
graphs}\label{sec-multimodulesandgraphs}

We
 recall here
   symmetrized multiplicity modules and invariants of colored plane graphs introduced in   \cite{TVi}.
  Throughout this section, we fix a    pivotal $\kk$-additive  monoidal category $\cc$   such that
  $\End_\cc(\un)=\kk$.

\subsection{Multiplicity modules}\label{muliplicity
modules}
 A \emph{signed object} of~$\cc$  is a pair
$(U,\varepsilon)$ where $U\in \cc $ and $\varepsilon\in\{+,-\}$. For
a signed object $(U,\varepsilon)$ of $\cc$,  we  set
$U^\varepsilon=U$ if $\varepsilon=+$ and $U^\varepsilon=U^*$ if
$\varepsilon=-$.  A \emph{cyclic $\cc$\ti set} is a triple $(E,c\co
E \to \cc ,\varepsilon\co E \to \{+, - \})$, where $E$ is a totally
cyclically ordered finite set. In other words,  a cyclic $\cc$\ti
set is a totally cyclically ordered finite set whose elements are
labeled
  by signed objects of $\cc$.  For shortness, we
 sometimes write $E$ for   $(E,c
,\varepsilon )$.

Given a cyclic $\cc$\ti set $E=(E,c ,\varepsilon )$ and   $e\in E$,
  we can order $e =e_1<e_2< \cdots <e_n$   the
elements of $E$   starting from $e$ and using the
 given cyclic order  in $E$ (here $n=\# E$ is the number of elements
 of~$E$). Set
$$Z_e= c(e_1)^{\varepsilon(e_1)} \otimes \cdots \otimes
c(e_n)^{\varepsilon(e_n)} \in \cc \quad {\text {and}} \quad
H_e=\Hom_\cc (\un, Z_e   ).$$ By   \cite{TVi}, the  structure of a pivotal category
in $\cc$ determines a projective system of  \kt module isomorphisms
$\{H_e\approx H_{e'}\}_{e,e' \in E}$. For example, if in the
notation above $\varepsilon(e_1)=-$, then the isomorphism $H_{e_1}
\to H_{e_2}$ carries any $f\in H_{e_1}$ to $$(\rev_{c(e_1)} \otimes
\id_{Z_{e_2}}) (\id_{c(e_1)} \otimes f \otimes \id_{{c(e_1)}^*})
\lcoev_{c(e_1)} \in H_{e_2}.$$   If $\varepsilon(e_1)=+$, then the
isomorphism $H_{e_1} \to H_{e_2}$  carries
  $f\in H_{e_1}$ to $$(\lev_{c(e_1)} \otimes \id_{Z_{e_2}}) (\id_{{c(e_1)}^*}
\otimes f \otimes \id_{{c(e_1)}}) \rcoev_{c(e_1)}\in H_{e_2}.$$
 The projective limit of this system of isomorphisms
 $
H(E )=\underleftarrow{\lim} \, H_e $ is a \kt module depending only
on $E$. It is   equipped with   isomorphisms $ \{  H(E ) \to
H_e\}_{e\in E}$, called the \emph{cone isomorphisms}.

An isomorphism   of cyclic $\cc$\ti
  sets $\phi\co E \to  E' $
 is  a bijection  preserving the cyclic order and commuting with the maps to
 $\cc$ and $ \{+, - \}$. Such  a  $\phi$
induces    a  \kt isomorphism
 $H(\phi) \co H(E ) \to H(E' )$ in the obvious way.

The \emph{dual}    of  a cyclic $\cc$\ti set $E= (E,c,\varepsilon)$
is the cyclic $\cc$\ti set $ (E^\opp,c,-\varepsilon)$,
where~$E^\opp$ is~$E$ with opposite cyclic order. The  pivotality of
$\cc$   determines a \kt bilinear  pairing $\omega_E \co H(E^\opp )
\otimes H(E ) \to \kk$ where
  $\otimes=\otimes_{\kk}$ is the tensor product  over $\kk$.  A \emph{duality} between cyclic
$\cc$\ti sets~$ E $ and~$ E' $ is an isomorphism  of cyclic $\cc$\ti
sets $ \phi:
 E' \to E^\opp  $. Such a $ \phi$ induces a \kt bilinear  pairing
$ \omega_{E} \circ (H(\phi)  \otimes \id)\co H(E' ) \otimes H(E )
\to \kk$.

\subsection{Colored graphs}\label{sect-graph}  By a \emph{graph}, we mean
a finite graph  without isolated vertices.     We allow multiple
edges with the same endpoints and loops (edges with both endpoints
in the same vertex). A graph   is \emph{oriented}, if all its edges
are oriented.

Let $\Sigma$ be an oriented surface. By a \emph{graph} in~$\Sigma$,
we mean a  graph embedded in~$\Sigma$.   A vertex $v$ of a
  graph~${A} \subset \Sigma$ determines a cyclically ordered  set $A_v$
consisting of the   half-edges of~$A$ incident to~$v$ with cyclic
order induced by the opposite orientation of~$\Sigma$.  If $A$ is
oriented, then  we  have a map $\varepsilon_v \co A_v \to \{+,-\}$
assigning $+$ to the half-edges oriented towards $v$ and $-$ to the
half-edges oriented away from $v$.

  A {\it
$\cc$-colored graph} in $  \Sigma$ is an oriented  graph   in
$\Sigma$ whose every edge     is labeled with an object of $\cc $
called the \emph{color} of this edge. A vertex $v$ of a
$\cc$-colored
   graph ${A}\subset \Sigma$ determines a cyclic  $\cc$\ti set $A_v=(A_v, c_v, \varepsilon_v )$,
  where $c_v$ assigns to each half-edge the corresponding color.   Set $H_v(A)=H(A_v )$ and
 $H(A)=\otimes_v \, H_v(A) $, where $v$ runs over all vertices of~$A$. To
stress the role of~$\Sigma$, we   sometimes write $H_v(A;\Sigma)$
for~$H_v(A)$ and $H(A;\Sigma)$ for~$H(A)$.

Consider now the case $\Sigma=\R^2$ with counterclockwise
orientation.   Every $\cc$-colored graph $A$ in $\R^2$ determines a
vector $\inv_\cc ({A}) \in H({A})^\star=\Hom_\kk(H({A}),\kk)$, see
\cite{TVi}.  The idea behind the definition of $\inv_\cc ({A})$  is
to deform $A$ in the plane so that in a neighborhood of every vertex
$v$ all half-edges incident to $v$  lie above $v$ with respect to
the second coordinate  in the plane. For each $v$, pick any
$\alpha_v \in H_v({A})$ and replace $v$ by a box   colored with the
image of $ \alpha_v $ under the corresponding cone isomorphism. This
transforms $A$ into a planar diagram formed by colored edges and
colored boxes. Such a diagram determines, through composition and
tensor product of morphisms in $\cc$ (and the use of the left and
right evaluation/co-evaluation morphisms), an element of $
\End_\cc(\un)=\kk$. By linear extension, this procedure defines a
vector $\inv_\cc({A})\in H({A})^\star$. The key property of this
vector is the independence   from the auxiliary choices. Moreover,
both
  $H({A})$ and   $\inv_\cc ({A})$ are preserved
under color-preserving isotopies of $A$ in $\R^2$.

For example, if $A=S^1\subset \R^2$ is the unit circle  with  one
vertex $v=(1,0)$ and one edge oriented clockwise and colored with
$U\in  \cc $, then $A_v$ consists of two elements labeled by
$(U,+)$,
  $(U, -)$ and  $H({A})=H_v({A})\cong \Hom_{\cc} (\un, U^* \otimes U)$. Here
 $\inv_\cc({A})(\alpha)= \lev_U   \alpha \in \End_\cc(\un)=\kk$ for all $\alpha\in   H({A}) $.

If   $\cc$ is   pre-fusion, then the \kt modules
associated in Section \ref{muliplicity modules} with cyclic
$\cc$-sets are free of finite rank and the
 pairings  defined at the end of that section   are non-degenerate
 (cf.\@ \cite{TVi}, Lemmas 2.3 and  4.1.a).
Note also that every $\cc$-colored graph $A $ in an oriented surface
$ \Sigma$ determines a $\cc$-colored graph   $A^\opp$  in~$-\Sigma$
obtained by reversing orientation  in $\Sigma$ and in  all edges
of~$A$ while keeping the colors of the edges. The
  cyclic $\cc$-sets determined by a  vertex $v$ of $A$ and $A^\opp$
  are dual. If $\cc$ is pre-fusion, then  $H_v(A^\opp;-\Sigma)= H_v(A;\Sigma)^\star $ and $ H (A^\opp;-\Sigma)= H
  (A;\Sigma)^\star$.

\subsection{The   spherical  case}\label{spherical}    If $\cc$
is spherical,   then the invariant $\inv_\cc$ of  $\cc$-colored
graphs in $\RR^2$ generalizes to   $\cc$-colored
 graphs
in the 2-sphere
 $S^2=\R^2 \cup \{\infty\}$  (the orientation of~$S^2$ extends the
counterclockwise orientation in $\R^2$).  Indeed, pushing a
$\cc$-colored
 graph $A\subset S^2$
away from $
 \infty $, we obtain   a $\cc$-colored   graph $A_0$ in $
\R^2$.   The sphericity of $\cc$ ensures that  $$\inv_\cc({A})
=\inv_\cc(A_0)\in H(A_0)^\star=H({A})^\star$$ is a well defined
isotopy invariant of $A$. We state a few simple properties of
$\inv_\cc$.

\begin{enumerate}
\labeli
\item (Naturality) If a $\cc$-colored graph $ {A}'  \subset S^2$ is obtained from a
$\cc$-colored graph $A \subset S^2$ through the replacement of
the color $U $ of an edge by an isomorphic object $U'$, then any
isomorphism $\phi \co U\to U'$ induces an isomorphism $\hat \phi
\co H({A})\to H( {A}' )$ in the obvious way and ${\hat
\phi}^\star (\inv_\cc( {A}' ))=\inv_\cc({A})$.

\item If   an edge~$e$ of a $\cc$-colored graph $A \subset S^2$ is
colored by $\un$ and   the endpoints of~$e$ are also  endpoints
 of other edges of~$A$, then $\inv_\cc({A})=\inv_\cc(A\setminus
 \Int(e))$.

\item If a $\cc$-colored graph ${A}' \subset S^2$ is obtained from a
$\cc$-colored graph $A \subset S^2$ through the replacement of
the color $U $ of an edge $e$  by $U^*$ and the reversion of the
orientation of $e$, then there is an isomorphism $ H({A})\to
H({A}')$  such that the dual isomorphism carries $
\inv_\cc({A}')$ to $\inv_\cc({A})$.

\item We have $H(A\amalg {A}')=H({A}) \otimes H ({A}')$ and
$\inv_\cc(A\amalg {A}')=\inv_\cc({A}) \otimes \inv_\cc({A}')$
for any disjoint $\cc$-colored graphs $A, {A}'\subset S^2$.
\end{enumerate}


%

\section{Skeletons and maps}\label{sec-skeletonsstatesums}

 We recall  the notions of stratified
   2-polyhedra and skeletons of 3-manifolds following
   \cite{TVi}, Section 6. Then we explain how to encode maps from
   closed 3-manifolds to $\XX=K(G,1)$ in terms of labelings of   skeletons.

\subsection{Stratified 2-polyhedra}\label{strastra} By an
\emph{arc} in a topological space~$P$, we mean the image of a path
$\alpha\co [0,1]\to P$ which is an embedding except that
  possibly $\alpha(0)=\alpha(1)$  (i.e., arcs may be loops.) The
  points $\alpha(0)$, $\alpha(1)$ are the \emph{endpoints}   and the set $\alpha((0,1)) $ is the \emph{interior} of the
  arc. By a  \emph{2-polyhedron}, we mean a compact topological space that can be triangulated
using only simplices of dimensions $0$, $1$, and $2$.
For a 2-polyhedron $P$, denote by $\Int(P)$ the subspace of~$P$
consisting of  all points having a neighborhood homeomorphic to
$\RR^2$. Clearly, $\Int(P)$ is an (open) 2-manifold.

   Consider a   2-polyhedron~$P$  endowed with a
    finite     set of arcs~$E $    such that
\begin{enumerate}
  \labela
    \item different arcs in~$E$ may meet only at their endpoints;
    \item   $ P \setminus   \cup_{e\in E} \, e \subset \Int (P)$;
    \item  $ P \setminus   \cup_{e\in E} \, e$ is dense in~$P$.
\end{enumerate}
The arcs of~$E$ are called  \emph{edges} of~$P$ and their endpoints
    are called \emph{vertices} of~$P$.   The vertices and
edges  of $P$ form a
    graph $P^{(1)}= \cup_{e\in E} \, e  $.
Cutting $P$ along  $P^{(1)}$, we obtain a compact surface
${\widetilde P}$ with interior $P \setminus P^{(1)}$. The
polyhedron~$P$ can be recovered by gluing~${\widetilde P}$
to~$P^{(1)}$ along a   map $ p\co
\partial {\widetilde P} \to P^{(1)} $. Condition (c) ensures the surjectivity of~$p$. We call the pair $(P, E)$  (or, shorter,~$P$)   a \emph{stratified 2-polyhedron}
if
 the set $p^{-1}({\text {the set of   vertices of}} \, P)$ is finite
and each component of the complement of this set in $\partial
{\widetilde P} $ is mapped homeomorphically onto the interior of an
edge of~$P$.


For a   stratified 2-polyhedron~$P$, the connected components
of~${\widetilde P}$ are called \emph{regions} of~$P$. The set
$\Reg(P) $ of the regions of~$P$ is finite.  For a vertex   $x$
of~$P$,  a \emph{branch}   of~$P$ at~$x$ is
 a germ  at~$x$ of a region of~$P$ adjacent to~$x$.    The set of branches of~$P$ at~$x$ is finite and
 non-empty. The branches
of $P$ at $x$ bijectively correspond to  the elements of the set~$
p^{-1} (x)$, where $ p\co
\partial {\widetilde P} \to P^{(1)} $ is the map above.
 Similarly,    a \emph{branch}   of~$P$ at an edge~$e$ of~$P$ is
 a germ  at~$e$ of a region of~$P$ adjacent to~$e$.
  The set of
branches of~$P$ at~$e$ is  denoted~$P_e$.   This set   is finite and
non-empty. There is a natural bijection between~$P_e$ and  the set
of connected components of $ p^{-1} ({\text {interior of}}\, e) $.
The number of elements of $P_e$ is the \emph{valence} of~$e$. The
edges of~$P$ of valence 1 and their vertices form a graph called the
\emph{boundary} of~$P$ and denoted~$\partial P$. We say that $P$ is
\emph{orientable} (resp.\@ \emph{oriented}) if all regions of~$ {
P}$ are orientable (resp.\@ oriented).

\subsection{Skeletons of    3-manifolds}\label{sect-skeletons} Let $M$ be a closed   oriented 3-dimensional manifold.
 A \emph{skeleton} of   $M$ is
an oriented    stratified 2-polyhedron $P\subset M$ such that
$\partial P=\emptyset$ and  $M\setminus P$ is a disjoint union of
open 3-balls. The components  of $M\setminus P$ are called {\it
$P$-balls}. An example of a skeleton of~$M$ is provided by the
(oriented) 2-skeleton~$t^{(2)}$ of a triangulation~$t$ of~$M$, where
the edges of~$t^{(2)}$ are the edges of~$t$.

We now analyze    regular neighborhoods of edges and vertices of a
skeleton $P \subset M$.   Pick an edge $e$ of $P$ and orient it in
an arbitrary way. The orientations  of~$e$  and~$M$ determine a
positive direction on a small loop in $  M \setminus e$
encircling~$e$ so that the linking number of this loop with $e$ is
$+1$. This direction induces a
   cyclic order on the set $P_e$ of   branches of~$P$ at~$e$.
For a branch $b\in P_e$, set   $\varepsilon_e (b)=+$ if the
orientation of $e$ is compatible with  the orientation
 of $b$   induced by that of the ambient region of $P$ and set $\varepsilon_e (b)=-$ otherwise. This gives a   map $\varepsilon_e \co P_e \to \{+, -\}$.
 When orientation  of $e$ is reversed, the
  cyclic order  on $P_e$ is reversed and     $ \varepsilon_e$ is multiplied by $-$.

Any  vertex   $v$ of   $P  $ has a closed ball neighborhood $B_v
\subset M$ such that $\Gamma_v=P\cap \partial B_v$ is a  non-empty
graph and
  $ P\cap B_v $ is the cone over~$  \Gamma_v $.  The vertices of $\Gamma_v$ are the intersections of $\partial B_v$ with the
half-edges of $P$ incident to~$v$. Similarly, each  edge $\alpha$ of
$\Gamma_v$ is the intersection  of $\partial B_v$ with a branch
$b_\alpha$
  of $P$ at $v$. We  endow   $\alpha$  with orientation induced by that
  of $b_\alpha$ restricted to
$b_\alpha \setminus \Int (B_v)$.
  We
   endow~$\partial B_v\approx S^2$ with orientation induced
by that of~$M$ restricted to $M\setminus \Int(B_v)$.  In this way,
$\Gamma_v$ becomes an oriented graph in the oriented 2-sphere
$\partial B_v$. The pair
 $(\partial B_v, \Gamma_v)$ is   the \emph{link} of~$v$ in~$(M,P)$.
 It
   is well defined  up to orientation preserving  homeomorphism.
   Note that the condition $\partial P=\emptyset$ implies that every
vertex of $\Gamma_v$ is incident to at least two half-edges of
$\Gamma_v$.

\subsection{$G$-labelings}\label{sect-skeletons=} Let $P$ be a skeleton of a closed   oriented 3-dimensional manifold $M$.  A {\it $G$-labeling} of $P$ is a map $  \ell \co \Reg (P) \to G$
such that for every   edge $e$ of $P$ the labels of the adjacent
branches $b_1, \ldots , b_n$ of $P$, enumerated in the cyclic order
determined by an  orientation of $e$,  satisfy the following product
condition:
\begin{equation}\label{labelidentity} \prod_{i=1}^n \ell(b_i)^{\varepsilon_e (b_i)}=1,\end{equation}
where the map $\varepsilon_e:P_e\to \{+,-\}$ is determined by the
same orientation of $e$. Note that if \eqref{labelidentity} holds
for one orientation of $e$, then it holds also for the opposite
orientation.

 The $G$-labelings   of $P$ determine    homotopy classes of
maps  $ M  \to \XX=K(G,1)$ as follows.  Pick  a point, called the
{\it center} in every $P$-ball (i.e., in every component of
$M\setminus P$). Each region $r$ of $\Reg(P)$ is adjacent to two
(possibly coinciding) $P$-balls. Pick an arc $\alpha_r$ in $ M$
whose endpoints are the centers of these balls and whose interior
meets $P$ transversely in a single point lying in $r$. We choose the
arcs $\{\alpha_r\}_r$ so that they meet only in the endpoints and
orient them so that the intersection number $\alpha_r \cdot r=r\cdot
\alpha_r$ is $+1$ for all $r$.

 \begin{lem}\label{lem-arcs}   For each $G$-labeling $\ell$ of $P$
 there is a map $f_\ell\co M\to \XX$ carrying the centers of the $P$-balls to  the base point $x\in \XX$ and carrying  $\alpha_r$ to a
 loop in $\XX$ representing $\ell(r)\in G=\pi_1(\XX,x)$ for all $r\in
 \Reg (P)$. The homotopy class of $f_\ell$ depends only on $P$ and $\ell$.
 Any map $M\to \XX$ is homotopic to $f_\ell$ for some $G$-coloring
 $\ell$ of $P$.
\end{lem}

\begin{proof} Consider first the case where all
regions of $P$ are disks. Then $P$ determines a dual
CW-decomposition $P^*$ of $M$, whose 0-cells are the centers of the
$P$-balls, the 1-cells are the arcs $\{\alpha_r\}_r$, the 2-cells
are  meridional disks of the edges of $P^{(1)}$, and the 3-cells are
ball neighborhoods   in $M$ of the vertices of $P$. Consider a map
from the 1-skeleton of $P^*$ to $\XX$ carrying all 0-cells   to $x$
and carrying each $\alpha_r$ to a
 loop in $\XX$ representing $\ell(r) $. The product
condition \eqref{labelidentity} ensures that
 this map extends to the 2-skeleton of $P^*$. The equality $\pi_2(\XX)=0$  ensures  that there is a further extension to  a map $f_\ell:  M\to \XX$.
 The independence of $f_\ell$ of  the choice of   $\{\alpha_r\}_r$ follows from
 the fact that any two such systems  of arcs  are homotopic in $M$ relative to the centers of the $P$-balls. The independence of the choice
 of the centers and the last claim  of the lemma are straightforward.

In the general case, each   region $r$ of $P$ is  a disk with holes
and may be collapsed onto a wedge of circles $W_r \subset r
\setminus \partial r$ based at the point $\alpha_r\cap r$. Let $S_r$
be   obtained from $W_r\times [-1,1]$ by contracting the sets
$W_r\times \{-1\}$ and $W_r\times \{1\}$ into points called the
vertices of $S_r$. (If the two $P$-balls adjacent to $r$ are equal,
then we   identify the two vertices of $S_r$.)  We embed $S_r$ into
$M$ as a union of two cones with base $W_r$ and with the cone points
in the centers of the $P$-balls adjacent to $r$. We can assume that
$\alpha_r \subset S_r$ and that $S_r$ does not meet $S_{r'}$ for
$r\neq r'$ except possibly  in the vertices. Then the pair $(M,
\cup_r \, S_r)$ has a relative CW-decomposition whose only cells are
   meridional disks of the edges of $P$ and ball neighborhoods  in $M$ of the vertices of $P$. The projection $W_r\times [-1,1]\to [-1,1]$ induces a
    retraction $S_r\to \alpha_r$. Composing with loops $\alpha_r\to \XX$ representing $\ell(r)$ we obtain   a map
   $\cup_r \, S_r\to \XX$ carrying the centers of the $P$-balls to   $x$ and carrying  each $\alpha_r$ to a
 loop in $\XX$ representing $\ell(r) $.   The rest of the proof goes as in the previous paragraph.
 To prove the last claim of the lemma it is useful to note that  the inclusion  $W_r \to M$   is null-homotopic for all $r\in \Reg(P)$. \end{proof}

The maps from the set $\pi_0(M\setminus P)$ of $P$-balls to $G$
form a group   with respect to pointwise multiplication. This group
is called the \emph{gauge group of $P$} and denoted
by~$\mathcal{G}_P$. The group $\mathcal{G}_P$ acts (on the left) on
the set of $G$-labelings of~$P$: for  $\lambda \in \mathcal{G}_P$, a
$G$-labeling $ \ell \co \Reg (P) \to G$, and a region $r$ of $P$,
$$
(\lambda  \ell)(r)=\lambda(r_-)\, \ell(r)\, \lambda(r_+)^{-1}
$$
  where $r_\pm$ are the $P$-balls (possibly coinciding)
adjacent to~$r$ and  indexed so that $\alpha_r(0) \in r_-$ and
$\alpha_r(1) \in r_+$.   It is easy to see that $\lambda  \ell$ is a
$G$-labeling of $P$   determining the same homotopy class of maps
$M\to \XX$ as~$\ell$. Moreover, two $G$-labelings   of $P$ determine
the same homotopy class of maps $M \to \XX$ if and only if these
$G$-labelings  belong to the same $\mathcal{G}_P$-orbit.

 \subsection{$G$-skeletons}\label{sect-skeletons++} A {\it $G$-skeleton} of a $G$-manifold $M$ is a pair   (a skeleton $P$ of $M$,
  a $G$-labeling of $P$
representing the given homotopy class of maps $  M\to \XX$). For
brevity, such a pair will be often denoted by the same letter $P$ as
the underlying skeleton.  Lemma~\ref{lem-arcs} shows that any
skeleton of $M$   extends to a $G$-skeleton.

\section{State-sum invariants of  closed $G$-manifolds}\label{state-suminvariabnts ofclosed}

\subsection{The state-sum invariant}\label{sec-computat}  Fix a spherical $G$-fusion category $\cc$ such that
$\dim(\cc_1)\in \kk$ is invertible in the ground ring $\kk$.
    For any closed $G$-manifold $M$, we define a topological invariant $|M |_\cc \in \kk$ of   $M$.
This invariant is obtained as a state sum on a $G$-skeleton
$P=(P,\ell)$ of~$M$ as follows.  Pick a representative set
$I=\amalg_{g\in G}\, I_g$ of simple objects of $\cc$.
 Let
$\coul{P}$ be the set of maps  $c \co \Reg(P) \to I$ such that $c(r)\in
I_{\ell (r)}$ for all regions $r$ of $P$.  For $c\in \coul{P}$ and an
oriented edge $e$ of $P $, we have  a \kt module $H_c(e)=H(P_e)$,
where~$P_e$ is the set of branches of~$P$ at~$e$ turned into a
cyclic $\cc$\ti set as follows: the cyclic order and the map to
$\{\pm \}$ are as  in Section~\ref{sect-skeletons} and the
$\cc$-color of a branch $b\in P_e$ is the value of $c$ on the region
of~$P$ containing  $b$.  If $e^\opp$ is the same edge  with opposite
orientation,
 then   $P_{e^\opp}=(P_e)^\opp$. This induces a   duality between
 the modules
 $H_c(e)$, $H_c(e^\opp)$ and a
 contraction  homomorphism  $ \ast_e \co H_c(e)^\star \otimes H_c(e^\opp)^\star \to\kk$.

By Section~\ref{sect-skeletons},  the link of a   vertex $v \in P$
is an oriented    graph $\Gamma_v \subset \partial B_v\approx S^2$.
Given   $c \in \coul{P}$, we transform  $\Gamma_v$ into a $\cc$-colored
graph by coloring each edge of $\Gamma_v$ with the value of $c$ on
the region of $P$ containing this edge. Section~\ref{spherical}
yields a tensor $\inv_\cc (\Gamma_v) \in H (\Gamma_v)^*$. By
definition,
 $H (\Gamma_v)= \otimes_e\, H_c(e)$, where
 $e$ runs over all edges of $P$ incident to $v$ and oriented away from $v$ (an edge   with both endpoints in $v$ appears in this tensor product twice
 with opposite orientations). The tensor product $\otimes_v \,\inv_\cc (\Gamma_v)$ over all
vertices~$v$ of~$P$ is a vector in
 $\otimes_e \, H_c(e )^\star$, where   $e$
 runs over all oriented edges of $P$. Set
$\ast_P=\otimes_e \, \ast_e\co \otimes_e   H_c(e )^\star \to \kk$.

\begin{thm}\label{thm-state-3man}
Set
\begin{equation}\label{eq-simplstatesum+}|M|_\cc= (\dim(\cc_1))^{-\vert P\vert} \sum_{c \in \coul{P}} \,\,  \left ( \prod_{r \in \Reg(P)} (\dim c(r))^{\chi(r)} \right ) \,
  {\ast}_P ( \otimes_v \,\inv_\cc (\Gamma_v)) \in \kk,\end{equation}
where   $\vert P\vert$ is the number of   $ P$-balls  and $\chi $ is
the Euler characteristic. Then $|M|_\cc$ is a topological invariant
of the $G$-manifold $M$ independent of   the choice  of $I$ and $P$.
\end{thm}

This theorem generalizes Theorems 5.1 and 6.1 of \cite{TVi}
which produce a 3-manifold invariant $|\cdot |_{\mathcal D} $ from
any
  spherical  fusion category $\mathcal D$ whose dimension is
invertible in the ground ring (the case $G=\{1\}$).

We illustrate Theorem~\ref{thm-state-3man} with two examples. First,
if the given homotopy class of maps $M\to \XX$ includes the constant
map, then $|M|_\cc=|M|_{\cc_1}$. This follows from the definitions
because the constant map is represented by the constant labeling $ 1
\in G$. In particular, $|S^3|_\cc=|S^3|_{\cc_1}=(\dim(\cc_1))^{-1}$.
Secondly, $|(S^1\times S^2, f )|_\cc=1$ for any map $f:S^1 \times
S^2 \to \XX$. Indeed, pick a point $s\in S^1$ and a circle $L\subset
S^2$. The set $P=(\{s\} \times S^2 ) \cup (S^1\times L)$ is a
skeleton of $S^1\times S^2$ with one edge $\{s\}\times L$ and three
regions. We orient  the two disk regions of~$P$ lying in $S^2$
counterclockwise and orient the annulus region of~$P$ in an
arbitrary way. We label the annulus region with $1\in G$ and label
both disk regions with an element $g $ of $ G$ represented
 (up to conjugation and inversion) by the   restriction
of $f$ to $S^1\times \{pt\}$.   This $G$-labeling of $P$ represents
$f$. Formula~\eqref{eq-simplstatesum+}
 gives $$|(S^1\times S^2, f)|_\cc=(\dim(\cc_1))^{-2}\!\!\!\!\!\!\!\sum_{j \in I_1, \, k, l\in I_g } \!\!\!\!\!\!\!\dim(k)\dim(l)\,  \rank_\kk \,
 \Hom_\cc(\un,j \otimes k\otimes j^* \otimes l^*).$$
The right-hand side is equal to $1$ as easily follows from
Lemma~\ref{bubbleidentity}.

  Theorem~\ref{thm-state-3man} will be proved at
  the end of the section. We   first   recall the   moves on skeletons introduced in \cite{TVi} and lift them to   $G$-skeletons.

\subsection{Moves on skeletons}   Let $M$ be a closed oriented 3-dimensional manifold. We consider four   moves    $T_1-T_4$ on a
 skeleton $P \subset M$ transforming it into   a new skeleton $P'$ of~$M$, see Figure~\ref{fig-moves}.
  The \lq\lq phantom edge  move"~$T_1$ keeps~$P$ as a polyhedron and  adds one new edge connecting   distinct vertices of
  $P$; this edge is an arc in $P$ meeting
    $P^{(1)}$ solely at the endpoints and has   valence~2 in~$P'$.
The \lq\lq contraction move" $T_2$ collapses into a point an
 edge~$e$ of~$P$.
 This
  move is allowed only when the endpoints of $e$ are distinct and at least one   of them is the endpoint of some other
  edge.
The \lq\lq percolation  move" $T_3$  pushes a branch~$b$ of~$P$
through a vertex~$v$ of~$P$. The branch~$b$ is pushed across a small
disk~$D$ lying in
   another branch of~$P$ at~$v$  so that $D\cap P^{(1)}= \partial D\cap
   P^{(1)}=\{v\}$ and both these branches are adjacent to the same component of $M\setminus
   P$. The disk $D$ becomes a region of the resulting skeleton $P'$;
   all other regions of $P'$ correspond bijectively to the regions
   of $P$ in the obvious way.
The \lq\lq bubble  move" $T_4$  adds to~$P$ an embedded disk
$D_+\subset M$ such that   $D_+\cap P=\partial D_+ \subset
P\setminus P^{(1)}$,  the circle $\partial D_+$ bounds a disk~$D_-$
in $ P\setminus P^{(1)}$, and the 2-sphere $D_+\cup D_-$ bounds a
ball in $M$ meeting~$P$ precisely
   at~$D_-$.
   A  point of the circle $\partial D_+ $ is chosen as a vertex and the circle itself is viewed
   as an edge of the resulting skeleton  $P'$. The disks $D_+$ and $D_-$ become  regions of   $P'$;
   all other regions of $P'$ correspond bijectively to the regions
   of $P$.

   \begin{figure}[h, t]
\begin{center}
\psfrag{T}[Bc][Bc]{\scalebox{.9}{$T_1$}} \psfrag{L}[Bc][Bc]{\scalebox{.9}{$T_2$}}  \rsdraw{.45}{.9}{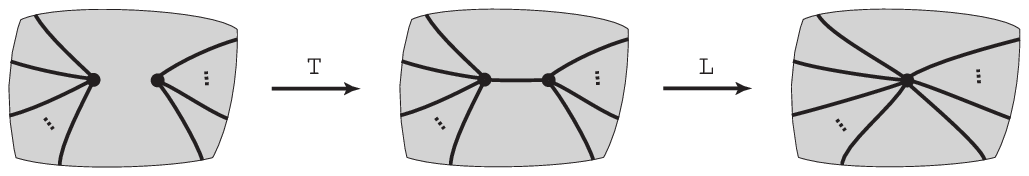} \\
[1.6em]  \psfrag{T}[Bc][Bc]{\scalebox{.9}{$T_3$}}  \rsdraw{.45}{.9}{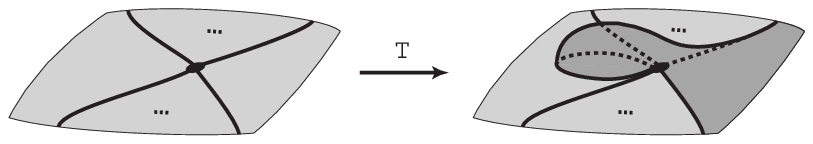} \\
[1.6em]  \psfrag{T}[Bc][Bc]{\scalebox{.9}{$T_4$}}  \rsdraw{.45}{.9}{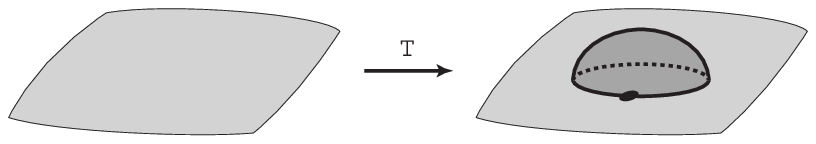}
\end{center}
\caption{Local moves on skeletons}
\label{fig-moves}
\end{figure}

In the  pictures of $T_1-T_4$ and in similar pictures below we
distinguish the \lq\lq small" regions   entirely contained in the
3-ball where the move proceeds and the \lq\lq big" regions   not
entirely contained in the 3-ball where the move proceeds. The moves
$T_1$, $T_2$ have no small regions,
 $T_3$ creates one small region $D$, and $T_4$ creates  two
small regions  $D_+$ and $D_-$. In the definition of   $T_1-T_4$  we
use the following orientation convention:  orientations of the big
regions   are preserved under the move  while   orientations of the
small regions may be arbitrary.

  The moves $T_1-T_4$ have  obvious inverses. The move   $T_1^{-1}$
 deletes a 2-valent edge~$e$ with distinct endpoints; this move is
   allowed only when both endpoints of~$e$ are endpoints of some other edges and the orientations of the two regions adjacent to~$e$ are compatible.

The  moves $T_1-T_4$   lift to $G$-labelings of skeletons of $M$ by
requiring that the labels of all big regions are preserved under the
moves. We impose no conditions on the labels of the small regions
except the product condition \eqref{labelidentity} which must hold
both before and after the move.  The labelings transform in a unique
way under $T_1^{\pm 1}, T_2^{\pm 1}, T_3^{\pm 1}, T_4^{-1} $. Under
$T_4$, the label of $D_+$ may be an arbitrary element of~$G$ and the
label of $D_-$ is then determined uniquely.  If $M$ is a
$G$-manifold, then each of these moves transforms a $G$-skeleton of
$M$ into a $G$-skeleton of $M$. These
  transformations of $G$-skeletons   are denoted by the same symbols $T_1^{\pm
1}-T_4^{\pm 1}$ and called {\it primary moves}. Label-preserving
ambient isotopies of $G$-skeletons in $M$ are also viewed as primary
moves.

  \begin{lem}\label{lem-momovesl}   Any
  two $G$-skeletons  of  a closed $G$-manifold $M$ can be related by a finite sequence of primary
moves.
\end{lem}

\begin{proof}  We first define several further moves on $G$-skeletons
of $M$. In these definitions, we   apply  the same orientation
  and labeling conventions as above.  For any
non-negative integers $m, n$ with $m+n\geq 1$, we define a
 move $T^{m,n}$  on   $G$-skeletons, see
Figure~\ref{fig-Tmn-move}. The move $T^{m,n}$  destroys $\max(m-1,
0)$ small regions and creates $\max(n-1, 0)$ small regions. The
labelings transform in a unique way under $T^{m,n}$.
 For $n=0$, this move is allowed only  when the
orientations of the top and bottom regions on the left are
compatible.    It is shown in \cite{TVi} that $T^{m,n}$ is a
composition of primary moves (and the same argument works in the
$G$-labeled case). The move inverse to $T^{m,n}$ is $T^{n,m}$.

\begin{figure}[h,t]
\begin{center}
\psfrag{T}[Bc][Bc]{\scalebox{1}{$T^{m,n}$}}
\psfrag{E}[cr][cr]{\scalebox{1.5}[3.1]{$\{$}}
\psfrag{L}[cl][cl]{\scalebox{1.5}[4.4]{$\}$}}
\psfrag{u}[cc][cc]{\scalebox{.9}{$m$}}
\psfrag{v}[cc][cc]{\scalebox{.9}{$n$}}
\rsdraw{.45}{.9}{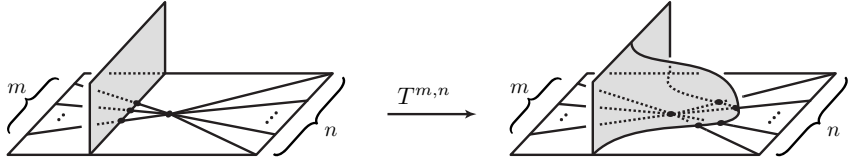}
\end{center}
\caption{The move $T^{m,n}$ on $G$-skeletons}
\label{fig-Tmn-move}
\end{figure}

In particular,  the moves $T^{2,0}$ and  $T^{0,2}$ push a branch of
a $G$-skeleton $P$ across a segment of $P^{(1)}$ containing a vertex
of valence 2. One can consider similar moves  $\widetilde T^{2,0}$
and $\widetilde T^{0,2}$ pushing   a branch of $P$ across a segment
of $P^{(1)}$   containing no vertices.  These   moves
decrease/increase  the number of vertices of the skeleton by 2.  The
moves $\widetilde T^{2,0}$ and $\widetilde T^{0,2}$ may be expanded
as compositions of  primary moves. Indeed, applying $T_2^{-1}$, we
can
  transform any point of $P^{(1)}$ into a vertex (keeping $P$
and $P^{(1)}$) and then use $T^{2,0}$ and $T^{0,2}$.

We need the modified version
  of the bubble move shown in Figure~\ref{fig-movesTT}.
It is easy to show that this move is a composition of   primary moves
(cf.\ \cite{TVi}, Section 7.2).

\begin{figure}[h, t]
\begin{center}
\psfrag{T}[Bc][Bc]{\scalebox{.9}{}}  \rsdraw{.45}{.9}{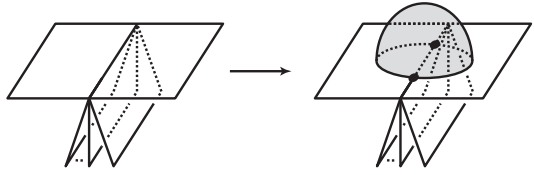}
\end{center}
\caption{A modified bubble move on $G$-skeletons}
\label{fig-movesTT}
\end{figure}

We now prove the following special case of the lemma.
\begin{claim}Let $P\subset M$ be a $G$-skeleton   of $M$ and
$r\in \Reg (P)$. Let $P_-$ be the same $G$-skeleton with both the
orientation and the label of $r$   inverted. Then there is   a
finite sequence of primary moves  transforming $P$ into $P_-$ (we
  write   $P\leadsto P_-$).
\end{claim}

\begin{proof}
It is enough to prove this  claim in the case where $r$ is a disk.
Indeed, if $r$ is not a disk, then we can  use $T_1$ to  add  new
edges to $P$ lying in $r$ and splitting $r$ into disks.
Consecutively reversing the orientation and the labels of these
disks and then removing the newly added edges by $T_1^{-1}$ we
obtain $P_-$. Therefore if our claim holds for
  disk regions of $G$-skeletons, then it also holds for $r$.

A similar argument shows that it is enough to consider the case
where   $r$ satisfies the following condition: $(\ast)$ the
restriction of the gluing  map $
\partial {\widetilde P} \to P^{(1)} $ (see Section \ref{strastra}) to    $\partial r$ is injective and all edges of $r$ (i.e., all edges of $P$ adjacent to $r$)
have valence $\geq 3$.   Indeed, adding  a bubble on each edge of
$r$ as in Figure \ref{fig-movesTT} and then pushing these bubbles
along the edges of $r$ and further across the vertices, we transform
$r$ into a smaller   region $r_0\subset r$ satisfying $(\ast)$.
Reversing the orientation and the label of $r_0$ and then removing
the newly added bubbles we obtain $P_-$. Therefore if our claim
holds for $r_0$, then it  holds for $r$.

Let now $r$ be a disk region of $P$ satisfying $(\ast)$ and let
$f\in G$ be the label of~$r$. Since $M$ and   $r$ are oriented, we
may speak about positive and negative   normal vectors on $r$.  Let
$r_{-} \subset M \setminus r $ be a 2-disk obtained by pushing~$r$
in the positive normal direction so that~$\partial r$ sweeps a
narrow annulus $A\subset P- \Int (r)$. It is understood that
$\partial A=
\partial r \amalg \partial r_{-}$ and $\partial r_{-}=r_{-}\cap P$. Then
 $\widehat P=P\cup r_{-}$ is a skeleton of~$M$ where the
regions of~$\widehat P$ contained in~$P$ receive the induced
orientation and
 the orientation of    $ r_{-}$ is opposite to that of~$r$.   Every region $R$ of $\widehat P$ distinct from  $r,r_{-}$ and not lying
  in $A$ is contained in a unique region of $P$.  We take the label of the latter region as the label of $R$.
  We endow $r,r_{-}$ with labels $h,h_-\in G$, respectively, such that $h h_-^{-1}=f$.
  The labels of the regions of $\widehat P$ lying in $A$ are determined uniquely by the product condition.
  This turns $\widehat P$ into a $G$-skeleton of $M$. We say that $\widehat P$ is obtained by
 {\it doubling} $r$.   We can transform~$\widehat P$ into~$P$ by
  moves of type $T^{m,n}, \widetilde T^{0,2}, \widetilde T^{2,0}$ pushing $ \partial r_{-}$ inside~$r$
  and an inverse bubble move eliminating the  bubble resulting from $r_-$.
 A similar elimination of~$ r$ transforms~$\widehat P$ into   $P_-$. Therefore~$P$ and~$P_-$ are related by primary
moves. This completes the proof of our claim.
\end{proof}
We can now  finish  the proof of   Lemma~\ref{lem-momovesl}. By
\cite{TVi}, Lemma 7.1, any two skeletons of $M$ can be related by a
finite sequence of primary moves. This   lifts to a sequence of
primary moves on $G$-skeletons. Therefore we need only to prove that
if $P_1=(P,\ell_1)$ and $P_2=(P,\ell_2)$ are two $G$-skeletons
of~$M$ with the same underlying skeleton~$P$, then there is a
sequence of primary moves   $P_1 \leadsto  P_2$.

By Section~\ref{sect-skeletons=},  the $G$-labelings $\ell_1$
and $\ell_2$ of $P$ lie in the same $\mathcal{G}_P$-orbit, where
$\mathcal{G}_P=\mathrm{Map}(\pi_0(M\setminus P),G)$ is the gauge
group of $P$. This group is generated by the maps  $\lambda_{g,b}\co
\pi_0(M\setminus P) \to G$, with $g\in G$ and $b$ a $P$-ball, where
$\lambda_{g,b}$ carries $b$ to $g$ and carries all other $P$-balls
to $1$. To prove Lemma~\ref{lem-momovesl}, it is enough to produce
for any $G$-labeling $\ell$ of $P$, any $g \in G$, and any $P$-ball
$b$, a sequence of primary moves    $(P,\ell) \leadsto
(P,\ell_{g,b}=\lambda_{g,b}  \ell)$.

%
%
%

Consider the bubble move $(P,\ell) \mapsto (P',\ell')$ attaching a
small disk inside $b$ to a region of $P$ adjacent to $b$.  Under
this move the ball $b$ splits into two balls: the small one (the
bubble) and the complementary one, $b'$. The   same bubble move
transforms   $(P,\ell_{g,b})$ into $(P',\ell'_{g,b'})$. Therefore if
there are primary moves transforming $(P',\ell')$ into
$(P',\ell'_{g,b'})$, then there are primary moves transforming
$(P,\ell)$ into $(P,\ell_{g,b})$. Similar arguments work  for   the
modified bubble move shown in Figure~\ref{fig-movesTT} and for the
moves $T^{m,n}$. Applying such moves inside $b$,  we can replace $P$
and $b$ with another pair still denoted $P$, $b$ such that the new
$P$-ball $b$ has one additional property: its closure $\overline b$
is a closed embedded 3-ball in $M$ with interior $b$. Then each
region of $P$ is adjacent to $b$ on one side or not at all.

Suppose   there is a  region $r_0$ of $P$   adjacent to $b$  such
that  the positive normal vectors on $r_0$ look inside $b$.
Consider the $G$-skeleton $(  P_-,  \ell_-)$ obtained from
$(P,\ell)$ by inverting the orientation and the label of $r_0$. By
the claim above, there are sequences of primary moves $(P,\ell)
\leadsto (  P_-,  \ell_-)$ and $(P,\ell_{g,b}) \leadsto (  P_-, (
\ell_{g,b})_-)$. Observe that $  ( \ell_{g,b})_- =(\ell_-)_{g,b}$.
Therefore if there is a sequence of primary moves
 $(  P_-,  \ell_-) \leadsto (  P_-,
(\ell_-)_{g,b})$, then there is a sequence of  primary moves
$(P,\ell)\leadsto (P, \ell_{g,b})$. Continuing by induction, we can
reduce ourselves to the case where   the orientation of all regions
of $P$ lying in the 2-sphere $\partial \overline b$ is induced by
that of $M$ restricted to $ \overline b$.

  We now produce a sequence of primary moves  $(P,\ell)\leadsto
(P,\ell_{g,b})$. Pick a region $r_0$ of $P$ adjacent to $b$. We
first apply a bubble move $(P, \ell)\mapsto (P', \ell')$ which adds
to $P$ a
  2-disk $D_+ \subset \overline b$  such that the circle  $\partial D_+$ lies in   $r_0$   and bounds a   2-disk $D_-\subset r_0$.
  We endow $D_-$ with the orientation induced by that of $r_0$ and orient $D_+$ so
   that $\partial D_+=\partial D_-$ in the category of  oriented manifolds.
  We label $D_-$ with $g^{-1} \ell (r_0) $   and $D_+$ with $g $
   (the orientations and the labels of all \lq\lq big" regions of $P'$ are
 the same as in $P$).  Next,
  we isotop  the disk $D_+$ in $b$ so that it  sweeps $b$ almost entirely   while its   boundary  slides along $\partial
\overline b$. We arrange that in the terminal position, $D'_{+}$,
of the moving disk  its boundary circle lies   in   $r_0 \setminus
D_-$
    and bounds there a   2-disk.   This isotopy of $D_+$ transforms the $G$-skeleton $(P', \ell')$ into a new $G$-skeleton $(P'',\ell'')$ via a
    sequence of moves    $T^{m,n}$, $\widetilde T^{2,0}$, and  $\widetilde T^{0,2}$. Under these moves, all regions of the
    intermediate skeletons lying in   $\partial \overline b$ are provided with orientation induced by that of $M$ restricted to $\overline  b$.
    This ensures that there are no orientation obstructions to the moves $  T^{2,0}$ and $\widetilde T^{2,0}$ that may appear in our sequence.
   Finally, the inverse bubble move,  removing $D'_{+}$, transforms  $(P'',\ell'')$  into
 $(P,\ell_{g,b})$.
\end{proof}

\subsection{Proof of   Theorem~\ref{thm-state-3man}} The state sum $|M|_\cc$ does not depend
on the choice of the representative set~$I$   by the naturality of
$\inv_{\cc}$ and of the contraction maps. Lemma~\ref{lem-momovesl}
shows that  to prove the rest of the theorem, we need only to prove
that $|M|_\cc$ is preserved under the primary moves $P \mapsto P'$.
This follows from the \lq\lq local invariance" which says that the
contribution of any
  $c\in \coul{P}$ to the state sum is equal to the sum of the contributions of all $c'\in \coul{P'}$   equal to $c$ on all big regions.
For the primary moves $T_1$, $T_2$, $T_3$,   this local invariance
was proved in    \cite{TVi}, Section 7.5.
 For the bubble move $T_4$, the local
invariance  follows from  Lemma \ref{bubbleidentity} where
  $U$ is the value of $c$ on the   region of $P$
  where the bubble is attached, $V=\un$,  and  $k,l $ are the values of $c'$ on the   disks $D_+$, $D_-$ created by
 the move. The   factor  $\dim (U)  \dim(\cc_1) $ is compensated by the change in
the number of components of $M\setminus P$  and in the Euler
 characteristic. We  use here the equality $\ast_e
 (\inv_{\cc} (\Gamma_v)) = N^{\un}_{U\otimes k\otimes l}$,  where
 $e$ and $v$ are respectively the   edge and the vertex forming the circle $\partial D_+= \partial D_-$.

\subsection{Remark}\label{rem-onemoreinv}   The right-hand side of
Formula~\eqref{eq-simplstatesum+}  is the product of $
(\dim(\cc_1))^{-\vert P\vert}$ and a certain sum which we denote
$\Sigma_\cc(P)$. The definition of $\Sigma_\cc(P)\in \kk$ does not
use the assumption that $ \dim(\cc_1) $ is invertible in~$\kk$ and
applies to an arbitrary spherical $G$-fusion category~$\cc$. This
allows us to generalize the invariant $ \vert
 M \vert_\cc$ of a closed $G$-manifold~$M$  to
 any such~$\cc$.  We use the theory of spines, see \cite{Mat1}. By a spine of $M$, we mean an oriented stratified 2-polyhedron $P\subset M$ such that $P$
  has at least 2 vertices,
 $P$ is locally
 homeomorphic to the cone over the 1-skeleton of a tetrahedron,  and  $M\setminus P$ is an open ball.   By
 \cite{Mat1}, $M$  has a   spine~$P$ and any two  spines of~$M$ can be related by  the   moves $T^{1,2}$, $T^{2,1}$
  in the class of  spines. The arguments above imply that   $\Sigma_\cc(P)$ is preserved under these moves.
  Therefore   $\vert \vert
 M\vert\vert_\cc= \Sigma_\cc(P)$   is a
 topological invariant of~$M$. If $\dim(\cc_1)$ is invertible, then $   \vert \vert
 M\vert\vert_\cc= \dim(\cc_1)\,  \vert
 M \vert_\cc$.

\section{The state-sum HQFT}\label{sec-TQFTfin}

 In generalization of
the state-sum invariant introduced in the previous section, we
derive from any spherical $G$-fusion category $\cc$ such that
$\dim(\cc_1)\in \kk$ is invertible   an HQFT $\vert \cdot \vert_\cc$
with target $\XX=K(G,1)$.

\subsection{Skeletons of $G$-surfaces}\label{sec-skesGGs} Let~$\Sigma$ be a pointed closed oriented
surface.        Recall that each component of $\Sigma$ has a  base
point and    $\Sigma_\bullet$ is the set of    the base points. A
\emph{skeleton} of~$\Sigma$ is an oriented graph $A\subset \Sigma $
such that all components of $\Sigma\setminus A$ are open disks, all
vertices of~$A$ have valence $\geq 2$, and $A\cap
\Sigma_\bullet=\emptyset$. For example, the vertices and the edges
of a triangulation of~$\Sigma$ (with an arbitrary orientation of the
edges) form a skeleton provided   the base points of $\Sigma$ lie
inside  the 2-faces.

A \emph{$G$-labeling} $\ell$ of a skeleton   $A$ of $ \Sigma$ is a
map
 from the set of edges of $A$ to $ G$ such that for any vertex
$v$ of $A$,
\begin{equation}\label{eq-conp}\prod_{a\in A_v} \ell(a)^{\varepsilon_v(a)}=1\,
,\end{equation} where the product is determined by the cyclic order
in $A_v$ and $\ell(a) \in G$ is the value of $\ell$ on  the edge of
$A$ containing the half-edge $a$ (see Section \ref{sect-graph} for
the definition of $A_v$ and $\varepsilon_v$).  A $G$-labeling $\ell$
of $A$ determines a homotopy class of maps $f_\ell:(\Sigma,
\Sigma_\bullet)\to (\XX,x)$ as follows  (cf.\ Section
\ref{sect-skeletons=}).  Pick a {\it central point} in each
component of $\Sigma \setminus A$  so that all base points of
$\Sigma $ are among these centers. Choose oriented arcs in $\Sigma$
dual to the edges of $A$ and connecting the  central points of the
adjacent regions (the interiors of the arcs must be disjoint and the
intersection number of each edge with the dual arc is $+1$). The map
$f_\ell$ carries all the central points to $x$ and carries the arcs
in question  to loops in $(\XX,x)$ representing the values of $\ell$
on the corresponding edges. Formula \eqref{eq-conp} ensures that
such a map $f_\ell$ exists.  It is clear that the homotopy class of
$f_\ell$ depends only on $A$ and $\ell$.  We call the pair $(A,
\ell)$ a {\it $G$-skeleton} of the $G$-surface $(\Sigma, f_\ell)$.

An appropriate choice of a $G$-labeling turns  any skeleton of a
$G$-surface   into a $G$-skeleton of this $G$-surface.

\subsection{Skeletons in dimension 3}\label{sec-skesdim333} We
 now extend the theory of  skeletons of closed $G$-manifolds  to  $G$-manifolds with
boundary. Given a compact oriented 3-manifold $M$ with pointed
boundary and a skeleton $A\subset
\partial M$, we define a
\emph{skeleton} of   $(M,A)$ to be an oriented
  stratified 2-polyhedron $P\subset M$ such that $P\cap
\partial M= \partial P$ and
\begin{enumerate}
  \labeli
    \item   $\partial P= A$  as graphs, i.e., $\partial P$ and $  A$ have the same vertices and edges;

\item for every vertex $v$ of $A$, there is a unique edge
$e_v$ of $P$   such that $v$  is an endpoint of $e_v$ and $e_v
\nsubseteq \partial M$; the edge $e_v$ is not a loop  and
$e_v\cap  \partial M =\{v\}$;

\item every edge $a$ of $A$  is an edge of~$P$  of valence 1; the
only region~$D_a$ of~$P$ adjacent to~$a$ is   a closed 2-disk
meeting $ \partial M$ precisely along $ a $; the orientation
of~$D_a$ is compatible with that of~$a$;

\item all components of $M\setminus P$ are
open or half-open 3-balls.
\end{enumerate}

 Note that  the boundary disks of the half-open
components of $M\setminus P$ are precisely  the components of  $
\partial M \setminus A$. Conditions (i)--(iii) imply that the
intersection of $P$ with a tubular neighborhood of~$\partial M$ in
$M$
  is  homeomorphic to
  $A\times [0,1]$. If an edge $e$ of $P$ has both endpoints in $\partial M$, then $e\subset A$ is an edge of $A$.

 Let now $M=(M,f\co (M, (\partial M)_\bullet )\to (\XX,x))$ be a
$G$-manifold   and let $A=(A, \ell)$ be a $G$-skeleton of the
$G$-surface $\partial M$.  Given a skeleton $P$  of   $(M,A)$,
consider  a
   map $\tilde \ell: \Reg (P)\to G$ such
that Formula~\eqref{labelidentity} holds for every edge of $P$ not
lying in $A$ and $\tilde \ell (D_a)=\ell(a)$ for every edge $a$ of
$A$.  The map $\tilde \ell$  determines a homotopy class of maps
$f_{\tilde \ell}\co (M, (\partial M)_\bullet )\to (\XX,x)$ as in
Section \ref{sect-skeletons=} where all points of $(\partial
M)_\bullet $ are chosen as the centers of the corresponding
components of $M\setminus P$. We say that  $\tilde \ell$ is a {\it
$G$-labeling} of  $P$ and     $(P,\tilde \ell)$ is a  {\it
$G$-skeleton of $(M,A)$} if $f_{\tilde \ell}=f$ is the given
homotopy class of maps. It is easy to see that  $(M,A)$
  has a skeleton  (cf.\ \cite{TVi}, Lemma 8.1) and every
skeleton of $(M,A)$ has a $G$-labeling turning it into a
$G$-skeleton.

The primary moves
 $T_1^{\pm 1} -  T_4^{\pm 1}$ defined above for  closed $G$-manifolds extend to $G$-skeletons  of
  $(M,A)$ in the obvious way. All these moves proceed inside 3-balls   in $\Int(M)$ and do not modify the boundary of the skeletons.   In particular, the move $T_1$
 adds
 an
 edge with  both endpoints in $\Int(M)$, the move~$T_2$ collapses an
 edge contained in $\Int(M)$, etc.
The action of the moves on the $G$-labelings is determined  by   the
requirement that the labels of the big regions are preserved under
the moves.    As in the  case of closed $G$-manifolds, the labelings
transform   uniquely under $T_1^{\pm 1}, T_2^{\pm 1}, T_3^{\pm 1},
T_4^{-1} $ and non-uniquely under $T_4$, where the label of $D_+$
may be an arbitrary element of~$G$.   These
 moves as well as  label-preserving
ambient isotopies of $G$-skeletons of $(M,A)$ keeping the boundary
pointwise are called {\it primary moves}. All primary  moves
transform $G$-skeletons of $(M,A)$ into  $G$-skeletons of $(M,A)$.

  \begin{lem}\label{lem-momovesl234}   Any
  two $G$-skeletons  of    $(M,A)$ can be related by a finite sequence of primary
moves in the class of $G$-skeletons  of    $(M,A)$.
\end{lem}

\begin{proof} The proof reproduces the proof of Lemma
\ref{lem-momovesl} with obvious changes. Instead of  Lemma 7.1 of
\cite{TVi}  we should use Lemma 8.1 of \cite{TVi} which says that
any two skeletons of $(M,A)$ can be related by primary moves in~$M$.
\end{proof}

\subsection{Invariants of  pairs $(M,A)$}\label{sec-Io3m+}
  Fix up to the end of Section~\ref{sec-TQFTfin}  a   spherical fusion
$G$-category~$\cc$ over~$\kk$ such that $ \dim(\cc_1)$  is
invertible in~$\kk$.   We shall derive from  $ \cc$   a
3-dimensional HQFT $\vert \cdot \vert_\cc$ with target $\XX=K(G,1)$.
The construction   proceeds in three steps described in this and the
next two sections.

Fix  a representative set $I=\amalg_{g\in G}\, I_g$ of simple
objects of $\cc$. By an {\it $I$-coloring}  of a $G$-skeleton $(A,
\ell)$ of a $G$-surface, we mean a map $c$ from the set of edges of
$A$ to $I$ such that $c(a)\in I_{\ell(a)}$ for all edges $a$ of $A$.
Note that an $I$-colored $G$-skeleton is   $\cc$-colored in the
sense of Section \ref{sect-graph} so that the definitions and
notation of that section apply.

For a  $G$-manifold  $M$  and an $I$-colored  $G$-skeleton $A =(A,
\ell, c)$ of $\partial M$,   we define   a topological invariant
$|M, A| \in \kk$ as follows. Pick a $G$-skeleton~$P=(P,\tilde \ell)$
of $(M,A)$. Let $\coul{P,c}$ be the set of all  maps  $\tilde c \co \Reg(P)
\to I$ such that $\tilde c(r)\in I_{\tilde \ell(r)}$ for all $r\in
\Reg(P)$ and $\tilde c(D_a)=c(a)$ for all edges $a$ of $A$. For
every $\tilde c \in \coul{P,c} $ and every oriented edge $e$ of $P $,
consider the  \kt module $H_{\tilde c}(e)=H(P_e)$, where~$P_e$ is
the set of branches of~$P$ at~$e$ turned into a cyclic $\cc$\ti set
as  in Section~\ref{sec-computat} (with $c$ replaced by~$\tilde c$).
Let $E_0$ be the set of oriented edges of~$P$ with both endpoints in
$\Int(M)$, and let~$E_\partial$ be the set of edges of~$P$ with
exactly one endpoint  in $\partial M$ oriented towards this
endpoint. Every vertex~$v$ of~$A$ is incident to  a unique edge
$e_v$ belonging to $ E_\partial$ and $ H_{\tilde c}(e_v) =
H_v(A^\opp;-\partial M)$. Therefore
 $$
 \otimes_{e\in E_\partial}\, H_{\tilde c}(e)^\star=\otimes_v \, H_v(A^\opp;-\partial M)^\star=H(A^\opp;-\partial M)^\star.
 $$
For  $e\in E_0$,  the equality $P_{e^\opp}=(P_e)^\opp$  induces a
duality between
 the modules
 $H_{\tilde c}(e)$, $H_{\tilde c}(e^\opp)$ and a
 contraction homomorphism  $   H_{\tilde c}(e)^\star \otimes H_{\tilde c}(e^\opp)^\star \to\kk$.  This contraction does not depend on the orientation of~$e$
  up to permutation of the factors. Applying these contractions, we obtain a homomorphism
$$\ast_P \co \otimes_{e\in E_0 \cup E_\partial}\, H_{\tilde c}(e )^\star \longrightarrow  \otimes_{e\in E_\partial}\, H_{\tilde c}(e)^\star=H(A^\opp;-\partial M)^\star .$$
 As in Section~\ref{sect-skeletons},
any  vertex   $v$ of~$P $ lying in $\Int(M)$ determines an oriented
graph~$\Gamma_v$ on $S^2$, and   $\tilde c$ turns $\Gamma_v$ into a
$\cc$-colored graph. Section~\ref{spherical}   yields a tensor
$\inv_\cc (\Gamma_v) \in H_{\tilde c}(\Gamma_v)^*$. Here
 $H_{\tilde c}(\Gamma_v)= \otimes_e\, H_{\tilde c}(e)$, where~$e$ runs over all edges of~$P$ incident to~$v$ and oriented away from~$v$. The tensor product $\otimes_v \,\inv_\cc (\Gamma_v)$ over all
vertices~$v$ of~$P$  lying in $\Int(M)$  is a vector in
 $\otimes_{e\in E_0 \cup E_\partial}\, H_{\tilde c}(e )^\star $.

\begin{thm}\label{thm-state-3man+r}
Set
\begin{equation*} |M,A |= (\dim(\cc_1))^{-\vert P\vert} \sum_{\tilde c \in \coul{P,c} } \,\,  \left ( \prod_{r \in \Reg(P)} (\dim \tilde c(r))^{\chi(r)} \right ) \,
  {\ast}_P ( \otimes_v \,\inv_\cc (\Gamma_v)) , \end{equation*}
where   $\vert P\vert$ is the number of components of $M\setminus P$
and  $\chi$ is the Euler characteristic. Then $|M,A |\in
H(A^\opp;-\partial M)^\star$ does not depend on the choice of~$I$
and $P$.
\end{thm}

\begin{proof} Since any two $G$-skeletons  of $(M,A)$ are related by primary moves, we need only to verify the invariance of
$|M,A|$ under these moves. This   is done exactly as in the proof of
Theorem~\ref{thm-state-3man}.
\end{proof}

Though there is a canonical isomorphism $H(A^\opp;-\partial
M)^\star\simeq H(A;\partial M)$ (see  Section~\ref{sect-graph}), it
is convenient  to view  $|M,A |$ as a vector in  $
H(A^\opp;-\partial M)^\star$.

\subsection{Functoriality}\label{sec-Io3m++}
Consider a   $G$-manifold $M$ whose boundary is a disjoint union of
two $G$-surfaces  $\Sigma_0$ and $\Sigma_1$. More precisely, we
assume that  $\partial M=(-\Sigma_0) \amalg \Sigma_1$ (as
$G$-surfaces). Given  an $I$-colored $G$-skeleton $A_i$ of
$\Sigma_i$ for $i=0,1$, we  form the $I$-colored $G$-skeleton
$A_0^\opp \cup A_1$ of $\partial M$.
 Theorem~\ref{thm-state-3man+r} yields
  a vector
  $$|M,A_0^\opp \cup A_1 | \in H(A_0 \cup  A_1^\opp, -\partial M)^\star= H(A_0, \Sigma_0)^\star \otimes H( A_1^\opp, -\Sigma_1)^\star.$$
The canonical isomorphism $H( A_1^\opp, -\Sigma_1)^\star\simeq H(
A_1, \Sigma_1)$ induces an isomorphism
$$
\Upsilon\co H(A_0, \Sigma_0)^\star \otimes H( A_1^\opp, -\Sigma_1)^\star \to \Hom_\kk\bigl (H(A_0, \Sigma_0),H( A_1, \Sigma_1)  \bigr ).
$$
Set $$|M, \Sigma_0, A_0,  \Sigma_1, A_1|= \frac{
(\dim(\cc_1))^{\vert A_1\vert}}{\dim (A_1)} \,  \Upsilon
\bigl(|M,A_0^\opp \cup A_1|\bigr)\co H(A_0;\Sigma_0) \to
H(A_1;\Sigma_1)
 ,$$
 where   $\vert A_1 \vert$ is the number of components of $\Sigma_1 \setminus A_1$ and $\dim ({A_1})$
is the product of the dimensions of the simple objects of~$\cc$
associated with the edges
 of~$A_1$ by the given $I$-coloring of $A_1$. By definition,  if
 $ \Sigma_1=\emptyset$, then $A_1 =\emptyset$, $\vert A_1\vert=0$, and $\dim
 ({A_1})=1$.

 \begin{lem}\label{lem-skeletons81}
Let $ M_i $ be a  $G$-manifold  with $\partial M_i=(-\Sigma_i)
\amalg \Sigma_{i+1}$, where $\Sigma_i, \Sigma_{i+1}$ are
$G$-surfaces and $i=0,1$. Let $ M$ be the $G$-manifold obtained by
gluing $M_0$ and $M_1$ along $\Sigma_1$ so that $\partial M
=(-\Sigma_0) \amalg \Sigma_{2}$. For any $I$-colored $G$-skeletons
$A_0\subset \Sigma_0$, $A_2\subset \Sigma_2$ and any
$G$-skeleton~$A_1$ of $\Sigma_1$,
$$|M, \Sigma_0, A_0,  \Sigma_2, A_2|=\sum_{c }
 \, |M, \Sigma_1, (A_1,c),  \Sigma_2, A_2| \circ |M, \Sigma_0, A_0,  \Sigma_1, (A_1,c) |,$$
 where $c$ runs over all $I$-colorings of~$A_1$.
\end{lem}

\begin{proof}   This  follows from the definitions since
 the union of a  $G$-skeleton,  $P_0$, of  $(M_0, A_0^\opp \cup A_1)$  with a  $G$-skeleton, $P_1$, of   $(M_1, A_1^\opp \cup A_2)$ is a
  $G$-skeleton, $P$, of  $(M, A_0^\opp \cup A_2)$ and $\vert
  P\vert = \vert
  P_0\vert+\vert
  P_1\vert -\vert A_1 \vert$. The   term $-\vert A_1 \vert$ explains the need for the factor $(\dim(\cc_1))^{\vert A_1\vert}$ in the definition of
  $ |M, \Sigma_0, A_0,  \Sigma_1, A_1|$. Similarly, given  a region  $r_1$  of $P_1$ and a region $r_2$ of $P_2$ adjacent to the same edge $e$ of $A_1$, the union $r=r_1 \cup r_2 \cup e$
  is a region of $P$ and $\chi (r)=\chi(r_1)+\chi (r_2)-1$. The   term $-1$
    explains the need for the factor $(\dim (A_1))^{-1}$ in
  the definition of $ |M, \Sigma_0, A_0,  \Sigma_1, A_1|$.
\end{proof}

\subsection{The  HQFT $\vert \cdot \vert_{\cc}$}\label{sec-TQFT-}
For a $G$-skeleton $A$ of a $G$-surface $\Sigma$, denote by
$\mathrm{Col} ({A})$   the set of all $I$-colorings of~$A$. Set
$$\vert A; \Sigma\vert^\circ =\oplus_{c\in \mathrm{Col}   ({A})} \,  H((A,c); \Sigma).$$
  Given a $G$-cobordism   $(M, h\co
(-\Sigma_0) \sqcup\Sigma_1 \simeq  \partial M)$ between $G$-surfaces
$\Sigma_0$ and $\Sigma_1$, we now  define for any $G$-skeletons
$A_0\subset \Sigma_0$ and $A_1\subset \Sigma_1$ a homomorphism
\begin{equation}\label{homom}
|M, \Sigma_0, A_0, \Sigma_1, A_1|^\circ \co \vert
A_0; \Sigma_0\vert^\circ\to \vert A_1; \Sigma_1\vert^\circ.
\end{equation}
For $i=0,1$ denote by $\Sigma'_i$ the $G$-surface
$h(\Sigma_i)\subset
\partial M$ with orientation induced by the one in $\Sigma_i$.  Then $A'_i= h(A_i)$ with the $G$-labeling induced by that of $A_i$ is a $G$-skeleton of $\Sigma'_i$.
Consider the homomorphism
\begin{equation}\label{eq-func-}
\sum_{\substack{c_0 \in \mathrm{Col}(A'_0)\\ c_1 \in
\mathrm{Col}(A'_1)}} |M, \Sigma'_0, (A'_0, c_0),  \Sigma'_1, (A'_1,
c_1)| \co \vert A'_0; \Sigma'_0\vert^\circ\to \vert A'_1;
\Sigma'_1\vert^\circ ,
\end{equation}
where
$$
|M, \Sigma'_0, (A'_0,
c_0), \Sigma'_1, (A'_1, c_1)|\co H((A'_0,c_0); \Sigma'_0) \to
H((A'_1,c_1); \Sigma'_1).
$$
Conjugating  \eqref{eq-func-} by the
obvious  isomorphisms   $\{ \vert A_i; \Sigma_i\vert^\circ  \cong
\vert
 A'_i  ; \Sigma'_i\vert^\circ\}_{i=0,1}$ induced by $h$, we obtain
the homomorphism \eqref{homom}.
 Lemma~\ref{lem-skeletons81} implies
that for any $G$-cobordisms $M_0, M_1, M$ as in this lemma and for
any $G$-skeletons $\{A_i \subset \Sigma_i\}_{i=0}^2 $,
\begin{equation}\label{eq-func} |M, \Sigma_0, A_0,  \Sigma_2, A_2|^\circ=|M,
\Sigma_1, A_1,  \Sigma_2, A_2|^\circ \circ  |M, \Sigma_0, A_0,
\Sigma_1, A_1|^\circ.\end{equation}

These constructions assign a finitely generated free \kt module to
every $G$-surface with distinguished $G$-skeleton and   a
homomorphism of these modules to every $G$-cobordism whose bases are
endowed with $G$-skeletons.
   This data   satisfies an
appropriate version of the axioms of an HQFT except one:   the
homomorphism associated with the cylinder over a $G$-surface,
generally speaking, is not the identity. There is a standard
procedure which transforms such a \lq \lq pseudo-HQFT"   into a
genuine
  HQFT and also gets rid of the skeletons. This procedure is described   in a similar situation  in
\cite{Tu0}, Section VII.3; we outline it in our setting.

Observe that if $A_0,A_1$ are two $G$-skeletons of a $G$-surface
$\Sigma$, then the cylinder $G$-cobordism $ \Sigma\times [0,1]$
gives a homomorphism
$$p(A_0,A_1)=|\Sigma\times [0,1], \Sigma \times \{0\}, A_0 \times \{0\},  \Sigma \times \{1\}, A_1\times \{1\}|^\circ \colon \vert A_0; \Sigma\vert^\circ \to \vert A_1; \Sigma\vert^\circ\, .$$
Formula~\eqref{eq-func} implies that  $p(A_0,A_2)=p(A_1,A_2)\,
p(A_0,A_1)$ for any $G$-skeletons $A_0$, $A_1$, $A_2$ of~$\Sigma$.
Taking $A_0=A_1=A_2$ we obtain that $p(A_0,A_0)$ is a projector onto
a direct summand $\vert A_0; \Sigma\vert $  of $\vert A_0;
\Sigma\vert^\circ $. Moreover,
  $p(A_0,A_1)$ maps
$\vert A_0; \Sigma\vert $ isomorphically onto $\vert A_1;
\Sigma\vert $. The finitely generated projective $\kk$-modules
$\{\vert A; \Sigma\vert\}_A$,
 where $A$ runs over all $G$-skeletons of $\Sigma$, and the
homomorphisms $\{p(A_0,A_1)\}_{A_0, A_1}$ form a projective system.
The projective limit of this system is a $\kk$-module   independent
of the choice of a $G$-skeleton of $\Sigma$, and we denote it by
$\vert \Sigma\vert_\cc $.   For each $G$-skeleton $A$ of~$\Sigma$,
we have a  \lq\lq cone isomorphism"    $\vert A; \Sigma\vert \cong
\vert \Sigma\vert_\cc$.   By convention, the empty surface
$\emptyset$   has a unique (empty) skeleton and   $\vert \emptyset
\vert_\cc =\kk$.

Any $G$-cobordism $(M,\Sigma_0, \Sigma_1)$   splits as a product of
a $G$-cobordism with a $G$-cylinder over $\Sigma_1$. Using this
splitting and Formula~\eqref{eq-func}, we obtain that the
homomorphism~\eqref{eq-func-} carries $ \vert \Sigma_0\vert_\cc
\cong \vert A_0; \Sigma_0\vert \subset \vert A_0;
\Sigma_0\vert^{\circ} $ into $  \vert \Sigma_1\vert_\cc  \cong \vert
A_1; \Sigma_1\vert \subset \vert A_1; \Sigma_1\vert^{\circ} $ for
any $G$-skeletons $A_0, A_1$ of $\Sigma_0, \Sigma_1$, respectively.
This gives a homomorphism   $|M, \Sigma_0, \Sigma_1 |_\cc \colon
\vert \Sigma_0\vert_\cc\to \vert \Sigma_1\vert_\cc$   independent of
the choice of $A_0$ and $A_1$. Moreover, two $G$-cobordisms
representing the same morphism $\varphi\co \Sigma_0\to \Sigma_1$ in
$\mathrm{Cob}_3^G$ give rise to the same homomorphism   $\vert
\varphi\vert_\cc \co \vert \Sigma_0\vert_\cc\to \vert
\Sigma_1\vert_\cc$. By construction, $\vert
\id_\Sigma\vert_\cc=\id_{\vert \Sigma \vert_\cc}$.    The assignment
   $\Sigma\mapsto \vert \Sigma \vert_\cc$, $\varphi \mapsto
|\varphi |_\cc $ defines a functor
$\vert\cdot\vert_\cc\co\mathrm{Cob}_3^G \to \mathrm{vect}_\kk$.
The   results above imply the following theorem.

\begin{thm}\label{thm-state-sum-HQFT}
The functor $\vert \cdot \vert_\cc$ is a 3-dimensional  HQFT   with target $\XX$.
\end{thm}

The HQFT $\vert \cdot \vert_\cc$ is called the {\it state sum HQFT}
derived from $\cc$.  Considered up to isomorphism, the HQFT $\vert
\cdot \vert_\cc$ does not depend on the choice of the representative
set $I$ of simple objects of $\cc$. For a  closed $G$-manifold~$M$,
the scalar $\vert M\vert_\cc \in \kk$ produced by this HQFT is
precisely the invariant  of Section~\ref{state-suminvariabnts
ofclosed}.

\subsection{Example} By \cite[Section I.2.1]{Tu1}, every  $\theta \in
H^3(G,\kk^*)=H^3(\XX;\kk^*)$   defines a 3-dimensional HQFT
$\tau^\theta$ with target $\XX$, called  a  primitive cohomological
HQFT. In particular,   $\tau^\theta(\Sigma)\cong \kk$ for any
$G$-surface $\Sigma$  and $\tau^\theta(M,g)=g^*(\theta)([M])$ for
any closed $G$-manifold $(M,g\co M \to \XX)$, where $[M]$ is the
fundamental class of $M$.
 On the other hand,  the spherical $G$-fusion category
$\mathrm{vect}_G^\theta$ of Example \ref{exaexa} defines a state-sum
HQFT $|\cdot|_{\mathrm{vect}_G^\theta}$ with target $\XX$. It can be
shown that these
  HQFTs are isomorphic:
$ \tau^\theta \cong |\cdot|_{\mathrm{vect}_G^\theta}$. This shows
that the HQFT $\tau^\theta$ may be fully computed via a state-sum on
$G$-skeletons.

\section*{Appendix. Push-forwards of categories and  HQFTs}\label{sect-push}

Let $\phi \co H \to G$ be a group  epimorphism  with finite kernel.
Every      $H$-category  $\cc=\oplus_{h \in H}\, \cc_h$ determines a
$G$-category $\phi_*(\cc)=\oplus_{g \in G} \, \phi_*(\cc)_g$ called
the \emph{push-forward  of~$\cc$}. By definition, $\phi_*(\cc)=\cc$
as  pivotal categories and $\phi_*(\cc)_g=\oplus_{h \in
\phi^{-1}(g)} \, \cc_h$ for all $g\in G$. If $\cc$ is $H$-fusion,
then $\phi_*(\cc)$ is $G$-fusion and $\dim(\phi_*(\cc)_1)= \gamma
\dim(\cc_1)$ where $ \gamma  $ is the order of $\Ker (\phi)$. If
$\cc$ is spherical, then so is $\phi_*(\cc)$.

Suppose   that   $\cc$ is a spherical $H$-fusion category and
$\gamma  \dim( \cc_1)\in \kk^*$. Then   $\cc$ and $\phi_*(\cc)$
determine   HQFTs $\vert \cdot \vert_{\cc}$ and $\vert \cdot
\vert_{\phi_*(\cc)}$ with targets $\YY=K(H,1)$ and~$\XX=K(G,1)$
respectively.  One can directly compute $\vert \cdot
\vert_{\phi_*(\cc)}$   from   $\vert \cdot \vert_{\cc}$ in terms of
the map $\widetilde  \phi:\YY\to \XX$ inducing  $\phi$ in $\pi_1$.
In particular, for any $G$-surface $(\Sigma,f\co \Sigma \to \XX)$
and for any closed connected $G$-manifold $(M, g\co M \to \XX)$,
$$
\vert\Sigma,f\vert_{\phi_*(\cc)}=\!\!\!\!\!\!\!\bigoplus_{ {\widetilde{f} \in [\Sigma,\YY],\,  \widetilde{\phi}\widetilde{f}=f}}
\!\!\!\!\!\!\! \vert\Sigma,\widetilde{f}\vert_{ \cc } \quad \text{and} \quad
\vert M,g\vert_{\phi_*(\cc)} =\gamma^{-1} \!\!\!\!\!\!\!\!\!\!\sum_{ {\widetilde{g} \in [M,\YY],\,  \widetilde{\phi}\widetilde{g}=g}} \!\!\!\!\!\!\!\! \vert M,\widetilde{g}\vert_{ \cc },
$$
where $[\Sigma,\YY]$ (resp.\@ $[M,\YY]$) denotes the set of homotopy
classes of maps $\Sigma \to \YY$ (resp.\@  $M\to \YY$).  It is
understood that the maps carry the set of base points of $\Sigma $
(resp.\@ a distinguished point of $  M $) to the base point of $\YY$
and the homotopies are constant on this set. The sums above are
finite because $\Ker (\phi)$ is finite. Similar formulas hold for
$G$-cobordisms.

We point out a special case of this construction. Any spherical
fusion category $\cc$
  can be viewed as  a spherical $\Gamma$-fusion category where $\Gamma$ is the graduator of $\cc$, see Section \ref{remarem}. Denote the resulting spherical $\Gamma$-fusion category by $\widetilde \cc$. Clearly, the group $\Gamma$ is finite and $\cc=\phi_*(\widetilde \cc) $ for
   the trivial   homomorphism $\phi\co\Gamma \to \{1\}$.
If $\dim(\cc)\in\kk^*$, then   $\cc$ gives rise to a state-sum TQFT
$\vert \cdot \vert_{\cc}$ and   $\widetilde \cc$ gives rise to a
state-sum HQFT $\vert \cdot \vert_{\widetilde \cc}$ with target
$K(\Gamma,1)$.
  For any  closed connected
oriented surface $\Sigma$ and any closed connected oriented
3-manifold $M$,
\begin{equation*}
|\Sigma|_\cc=\!\!\!\!\!\!\! \bigoplus_{f \co \pi_1(\Sigma) \to \Gamma} \!\!\!\!\!\!\! |\Sigma,f|_{\widetilde \cc } \qquad \text{and} \qquad
|M|_\cc=   \vert \Gamma\vert^{-1} \!\!\!\!\!\!\!\sum_{g \co \pi_1(M) \to \Gamma} \!\!\!\!\!\!\!  |M,g|_{\widetilde \cc }.
\end{equation*}
For more on this, see \cite{Petit}.

\end{document}